\documentclass[11pt,reqno]{amsart}
\usepackage{amsthm, amsfonts, amssymb, color}
 \usepackage{mathrsfs}
\usepackage{enumerate}
\usepackage{amsmath}
\usepackage{amssymb}
 \usepackage{amstext, amsxtra}
  \usepackage{txfonts}
 \usepackage[colorlinks, linkcolor=black, citecolor=blue, pagebackref, hypertexnames=false]{hyperref}

 \allowdisplaybreaks
\newtheorem{theorem}{Theorem}[section]
\newtheorem{proposition}[theorem]{Proposition}
\newtheorem{corollary}[theorem]{Corollary}
\newtheorem{lemma}[theorem]{Lemma}

\newtheorem{remark}[theorem]{Remark}

\newcommand\R{\mathbb{R}}

\newcommand{\la}{\lambda}

\newcommand{\supp}{{\rm supp}{\hspace{.05cm}}}

\numberwithin{equation}{section}
\theoremstyle{definition}

\title
[Schr\"{o}dinger-Poisson systems with critical nonlocal term]
 {The existence of positive least energy solutions for a class of Schr\"{o}dinger-Poisson systems involving critical nonlocal term with general nonlinearity}
\author{Liejun Shen \ and \ Xiaohua Yao}

\address{ Liejun Shen,  Hubei Key Laboratory of Mathematical Sciences and School of Mathematics and Statistics,
 Central China Normal University, Wuhan, 430079, P. R. China }
\email{liejunshen@sina.com}
\address{Xiaohua Yao, Hubei Key Laboratory of Mathematical Sciences and School of Mathematics and Statistics,
 Central China Normal University, Wuhan, 430079, P.R. China}
\email{yaoxiaohua@mail.ccnu.edu.cn }
\date{\today}

\begin{document}

\maketitle
\begin{abstract} The present study is concerned with the following Schr\"{o}dinger-Poisson system involving critical nonlocal term with general nonlinearity:
$$
\left\{%
\begin{array}{ll}
    -\Delta u+V(x)u- \phi |u|^3u= f(u), & x\in\R^3, \\
    -\Delta \phi= |u|^5, &  x\in\R^3,\\
\end{array}%
\right.
$$
Under certain assumptions on non-constant $V(x)$,
the existence of a positive least energy solution is obtained by using some new analytical skills and Poho\v{z}aev type manifold. In particular, the Ambrosetti-Rabinowitz type condition or monotonicity assumption on the nonlinearity is not necessary.
\end{abstract}


\section{Introduction and main results}
\setcounter{equation}{0}
In this paper, the existence of positive least energy solutions will be proved for
the following nonlinear Schr\"{o}dinger-Poisson systems involving critical nonlocal term:
\begin{equation}\label{mainequation1}
\left\{%
\begin{array}{ll}
    -\Delta u+V(x)u- \phi |u|^3u= f(u), & x\in\R^3, \\
    -\Delta \phi= |u|^5, &  x\in\R^3.\\
\end{array}%
\right.
\end{equation}
The potential $V(x)$ is assumed to verify the following conditions:
\vskip0.3cm
\begin{enumerate}[$(V_1)$]
\item \emph{$V(x)\in C(\R^3,\R)$ is weakly differentiable, $(x,\nabla V)$ and the best Sobolev constant $S$ (see \eqref{Sobolev} below) satisfy the following inequality:
\[
\bigg(\int_{\R^3}|(x,\nabla V)|^{\frac{3}{2}}dx\bigg)^{\frac{2}{3}}\leq S,
\]
where $(\cdot,\cdot)$ is the usual inner product in $\R^3$;}
\end{enumerate}
\vskip0.3cm
\begin{enumerate}[$(V_2)$]
\item \emph{there exists a positive constant $V_\infty$ such that for all $x\in\R^3$
$$
0\leq V(x)\leq \mathop{\lim\inf}_{|y|\to+\infty}V(y)=V_\infty<+\infty,
$$
where the two inequalities are strict in some subsets of positive Lebesgue measure.}
\end{enumerate}

 Since we are interested in positive solutions, without loss of generality, we assume that $f\in C(\R,\R)$ vanishes in $(-\infty,0)$ and satisfies the following conditions:
\begin{enumerate}[$(f_1)$]
\item  \emph{$f\in C(\R,\R^+)$ and $f(t)=o(t)$ as $t\to 0^+$, here $\R^+=[0,+\infty)$;}
\end{enumerate}
\vskip3mm
\begin{enumerate}[$(f_2)$]
  \item  \emph{$f$ has a subcritical growth at infinity, that is, $f(t)=o(t^5)$ as $t\to \infty$;}
\end{enumerate}
\vskip3mm
\begin{enumerate}[$(f_3)$]
  \item  \emph{there exist constants $\mu>0$ and $q\in(1,5)$ such that $F(t)\geq \mu t^{q+1}$, where $F(t)=\int_{0}^{t}f(s)ds$.}
\end{enumerate}

Due to the real physical meaning, the following Schr\"{o}dinger-Poisson system
\begin{equation}\label{introduction1}
 \left\{%
\begin{array}{ll}
    -\Delta u+V(x)u+\phi u=f(x,u), & x\in\R^3, \\
    -\Delta \phi=u^2, &  x\in\R^3,\\
\end{array}%
\right.
\end{equation}
has been studied extensively by many scholars in the last several decades. The system like \eqref{introduction1} firstly introduced by Benci and Fortunato \cite{Benci1} was used to
describe solitary waves for nonlinear Sch\"{o}rdinger type equations and look for the existence of standing waves interacting with an unknown electrostatic field. We refer the readers to \cite{Benci1,Benci2,P. L. Lions3,Markowich} and the references therein to get a more physical background of the system \eqref{introduction1}.

In recent years, by classical variational
methods, there are many interesting works about the existence and non-existence of positive solutions, positive
ground states, multiple solutions, sign-changing solutions and semiclassical states
to the system \eqref{introduction1} with different assumptions on the potential $V(x)$ and the nonlinearity $f(x,u)$ were established. If $V(x)\equiv 1$ and $f(x,u)=|u|^{p-1}u$, T. d'Aprile and D. Mugnai \cite{dAprile2} showed that the system \eqref{introduction1} has no nontrivial solutions when $p\leq 1$ or $p\geq 5$. For the case $4\leq p<6$, the existence of radial and non-radial solutions was studied in \cite{Coclite,dAprile1,dAprile3}. D. Ruiz \cite{Ruiz} proved the existence and nonexistence of nontrivial solutions when $1<p<5$. When $V(x)\equiv0$ and $f(x,u)=g(u)$,
A. Azzollini, P. d'Avenia and A. Pomponio \cite{Azzollini1} investigated the existence of nontrivial radial solutions when $\mu\in(0,\mu_0)$ for the following system
\begin{equation}\label{introduction4}
   \left\{%
\begin{array}{ll}
    -\Delta u+ \mu\phi u=g(u), & x\in\R^3, \\
    -\Delta \phi=\mu u^2, &  x\in\R^3.\\
\end{array}%
\right.
\end{equation}
under the conditions $g\in C(\R)$ and
\[
 (H_1)\ \  -\infty<\mathop{\lim\inf}_{t\to0}\frac{g(t)}{t}\leq\mathop{\lim\sup}_{t\to0}\frac{g(t)}{t}=-m<0;\ \ \ \ \ \ \ \ \ \ \ \ \ \ \ \   \ \ \ \ \ \ \ \ \ \ \ \   \ \ \ \ \ \ \ \ \ \ \ \ \ \ \ \   \ \ \ \ \ \ \ \ \ \ \ \ \ \ \ \
\]
\[
 (H_2)\ \  -\infty\leq\mathop{\lim\sup}_{t\to\infty}\frac{g(t)}{t^{{5}}}\leq0;\ \ \ \ \ \ \ \ \ \ \ \ \ \ \ \   \ \ \ \ \ \ \ \ \ \ \ \   \ \ \ \  \ \ \ \ \ \ \ \ \ \ \ \   \ \ \  \ \ \ \ \ \ \ \ \ \ \ \   \ \ \ \ \ \ \  \ \ \ \ \ \ \ \   \ \ \ \ \ \ \ \ \ \ \ \   \ \ \ \  \ \ \  \ \ \ \ \ \ \ \ \ \ \ \ \ \ \  \ \ \ \ \ \ \ \ \ \ \ \
\]
\[
 (H_3)\ \  \text{there exists}\ \ \xi>0\ \  \text{such that}\ \ \int_0^\xi g(t)dt>0.\ \ \ \ \ \ \ \ \ \ \ \ \ \ \ \   \ \ \ \ \ \ \ \ \ \ \ \   \ \ \ \ \ \  \ \ \ \ \ \ \ \   \ \ \ \ \ \ \ \ \ \ \ \   \ \ \ \  \ \ \ \ \ \ \ \
\]
We mention here that the hypotheses $(H_1)-(H_3)$ are the so-called Berestycki-Lions conditions, which were introduced
in H. Berestycki and P. L. Lions \cite{Berestycki} for the derivation of the ground state solution of \eqref{introduction4}. If $V(x)\not\equiv constant$ and $f(x,u)=|u|^{p-1}u+\mu |u|^{4}u$ with $2<p<5$, the existence of positive ground state was obtained by  Z. Liu and S. Guo \cite{Liu0}. By using superposition principle established by
N. Ackermann \cite{Ackermann}, the system \eqref{introduction1} with a periodic potential was studied by  J. Sun and S. Ma \cite{Sun}, where the existence of infinitely many geometrically distinct solutions was proved. For other related and important results, we refer the readers to \cite{Alves2,He,Huang,Jiang,Ruiz2,Zhang,Zhao} and their
references.

However, the results for the following general Schr\"{o}dinger-Poisson system
\begin{equation}\label{mmmmmm}
  \left\{%
\begin{array}{ll}
    -\Delta u+u+p \phi g(u)=f(x,u), & x\in\R^3, \\
    -\Delta \phi=2pG(u), &  x\in\R^3\\
\end{array}%
\right.
\end{equation}
are not so fruitful as the case $g(u)=u$ and $p\in\R$, where $|g(t)|\leq C(|t|+|t|^s)$ with $s\in (1,4)$, please see \cite{Azzollini3,Li2} for example. When $s=4$ in \eqref{mmmmmm},
A. Azzollini and P. d'Avenia \cite{Azzollini2} firstly studied the following Schr\"{o}dinger-Poisson system with critical nonlocal term
\begin{equation}\label{aaaaaaaa}
 \left\{%
\begin{array}{ll}
    -\Delta u=\mu u+p\phi|u|^3u, & x\in B_R, \\
    -\Delta \phi=p|u|^5, &  x\in B_R,\\
    u=\phi=0,& \text{on} \ \ \partial B_R.
\end{array}%
\right.
\end{equation}
 Note that although the second equation can be solved by a Green's function, the term $p|u|^5$ will result in a nonlocal critically growing nonlinearity in \eqref{aaaaaaaa}. After it, by assuming the conditions
 \vskip3mm
 \begin{enumerate}[$(h_1)$]
  \item  \emph{$h\in C(\R^+,\R^+)$ and $\lim_{t\to0^+}\frac{h(t)}{bt+t^5}=0$;}
\end{enumerate}
\vskip3mm
\begin{enumerate}[$(h_2)$]
  \item  \emph{$\lim_{t\to\infty}\frac{h(t)}{t^5}=0$;}
\end{enumerate}
\vskip3mm
\begin{enumerate}[$(h_3)$]
  \item \emph{there exist $r\in(4,6)$ and $A,B>0$ such that $H(t)\geq At^r-Bt^2$ for $t\geq 0$, here $H(t)=\int_{0}^{t}h(s)ds$,}
\end{enumerate}
  \vskip1mm
and using a monotonic trick introduced by L. Jeanjean \cite{Jeanjean}, F. Li, Y. Li and J. Shi \cite{Li} specially proved the following Schr\"{o}dinger-Poisson system
\[
  \left\{%
\begin{array}{ll}
    -\Delta u+bu-\phi |u|^3u=h(u), & x\in\R^3, \\
    -\Delta \phi=|u|^5, &  x\in\R^3,\\
\end{array}%
\right.
\]
possesses at least one positive radially symmetric solution when $b>0$ is a constant.

To the best of our knowledge, the Schrodinger-Poisson system with critical nonlocal term was rarely studied in \cite{Azzollini2,Li,Liu,Shen}. Inspired by the all works mentioned above, particularly, by the results in \cite{Li}, we try to prove the existence of positive least energy solution
for the system \eqref{mainequation1}.

Our main result is as follows:
\begin{theorem}\label{maintheorem1}
 Assume $V(x)$ and $f(x)$ satisfy the conditions $(V_1)-(V_2)$ and $(f_1)-(f_3)$, respectively. Then the system \eqref{mainequation1} admits a positive least energy solution $(u,\phi)\in H^1(\R^3)\times D^{1,2}(\R^3)$ for any $\mu>0$ with $q\in(3,5)$, or a sufficiently large $\mu>0$ with $q\in(1,3]$.
\end{theorem}

As a direct consequence of Theorem \ref{maintheorem1}, the following result is immediate:
\begin{corollary}\label{corollary}
 Under the assumptions of Theorem \ref{maintheorem1} with $(V_1)$ replaced by
 \begin{enumerate}[$(V_3)$]
   \item   $V(x)\in C(\R^3,\R)$ is weakly differentiable and there exists a constant $A\in(0,\frac{1}{4}]$  such that
\[
|(x,\nabla V)|\leq \frac{A}{|x|^2} \ \  \text{for almost}\ \  x \in\R^3 \backslash\{0\},
\]
where $(\cdot,\cdot)$ is the usual inner product in $\R^3$.
\end{enumerate}
Then the conclusion of Theorem \ref{maintheorem1} still holds.
\end{corollary}
\vskip0.4cm
\begin{remark}
 Compared with the results in F. Li, Y. Li and J. Shi \cite{Li}, there are four improvements in our paper: (i) We deal with the case that the potential $V(x)$ is not a constant. (ii) The restriction on the work space which is radially symmetric can be removed. (iii) Our approach to prove the existence of nontrivial solution is different. (iv) The nonlinear term are more general in our paper, and the solution obtained in Theorem \ref{maintheorem1} or Corollary \ref{corollary} can be confirmed as the least energy solution.
 \end{remark}
\vskip0.4cm

 \begin{remark}\label{potential}
On the potential in Theorem \ref{maintheorem1} and Corollary \ref{corollary}, we give some brief explanations:
\begin{enumerate}
\item To the best knowledge of us, the condition $(V_1)$ seems to be used for the first time to study Schr\"{o}dinger-Poisson systems.
  \item The hypothesis $V(x)\geq 0$ in $(V_2)$ can be replaced by: there exists a constant $\overline{C}>0$ such that
   \[
     \overline{C}\triangleq \inf_{u\in H^1(\R^3)\backslash\{0\}}\frac{\int_{\R^3}|\nabla u|^2+V(x)u^2dx}{\int_{\R^3}|u|^2dx}>0,
  \]
which is mainly used to ensure an equivalent norm to the usual norm in $H^1(\R^3)$, please see
 \cite{G. Li,Zhao0}.
  \item There are a great number of potential functions $V(x)$ satisfying the conditions $(V_1)$ and $(V_2)$. For example,
\[
V(x)=V_\infty\bigg(1-\frac{C}{1+|x|^{3}}\bigg).
\]
Similarly the functions
\[
V_1(x)=V_\infty-\frac{A}{1+|x|^2}
\]
verify the conditions $(V_2)$ and $(V_3)$.
\item In our paper, the range of $A$ in $(V_3)$ depends on the sharp constant of Hardy inequality \cite{Kufner}:
\begin{equation}\label{Hardy}
 \int_{\R^3}\frac{|u|^2}{|x|^2}dx\leq 4\int_{\R^3}|\nabla u|^2dx.
\end{equation}
\end{enumerate}
\end{remark}

Now we give our main idea for the proof of Theorem \ref{maintheorem1}. Although, it is easy to verify that the functional $I_{V}$ (see Section 2) possesses a Mountain-pass geometry in the usual way and then a $(PS)$ sequence can be obtained, it is difficult to prove its boundedness because the Ambrosetti-Rabinowitz type condition
($(AR)$ in short)
\vskip0.4cm
\begin{enumerate}[$(AR)$]
  \item \emph{There exists $\gamma>2$ such that $0 < \gamma F(t)\leq f(t)t$ for all $t \neq 0$}
\end{enumerate}
\vskip0.2cm
or the monotonicity assumption
\vskip0.1cm
\begin{enumerate}[$(M)$]
  \item \emph{The map $t \to
\frac{f(t)}{t}$ is positive for $t\neq 0$, strictly decreasing on $(-\infty,0)$ and
strictly increasing on $(0,+\infty)$}
\end{enumerate}
\vskip0.1cm
does not hold. Even a bounded $(PS)$ sequence can be established, the lack of compact Sobolev embedding and the condition $(M)$
lead to great difficulties in proving the functional satisfies the so-called $(PS)$ condition.
To overcome the above difficulties, motivated by \cite{Zhao2}, we use an indirect approach developed by L. Jeanjean \cite{Jeanjean2}. As a consequence, a bounded $(PS)_{c_\la}$ sequence for the functional $I_{V,\la}$ (see \eqref{IVK}) is obtained, we use a similar method presented in \cite{Jeanjean3} to recover the global compactness lemma and then to prove the $(PS)_{c_\la}$ condition. However before we success in showing the $(PS)_{c_\la}$ condition, we have to pull the energy level down below some specific critical level:
\[
c_\la<\frac{2}{5}\la^{-\frac{1}{4}}S^{\frac{3}{2}}
\]
because of the critical nonlocal term.

To apply the global compactness lemma mentioned above, first of all, we need to consider the existence of ground state solutions of the associated ``limit problem" of \eqref{mainequation1}, which is given as
\[
  \left\{%
\begin{array}{ll}
     -\Delta u+V_\infty u- \phi |u|^3u=\lambda f(u), & x\in\R^3, \\
     -\Delta \phi= |u|^5, &  x\in\R^3.\\
\end{array}%
\right.
\]
Since $V_\infty$ and $\lambda$ are constants, we just need to investigate the following equation
\begin{equation}\label{mainequation3}
  \left\{%
\begin{array}{ll}
     -\Delta u+ u- \phi |u|^3u=  f(u), & x\in\R^3, \\
     -\Delta \phi=  |u|^5, &  x\in\R^3,\\
\end{array}%
\right.
\end{equation}
for simplicity.\\

We obtain the following result:
\begin{theorem}\label{maintheorem2}
Assume $(f_1)-(f_3)$ hold, then the system \eqref{mainequation3} possesses a positive least energy solution $(u,\phi)\in H^1(\R^3)\times D^{1,2}(\R^3)$ for any $\mu>0$ with $q\in(3,5)$, or a sufficiently large $\mu>0$ with $q\in(1,3]$.
\end{theorem}

\begin{remark}\label{label}
To prove Theorem \ref{maintheorem2}, we adopt an approach proposed by L. Jeanjean \cite{Jeanjean} to construct
a $(PS)$ sequence $\{u_n\}$ which satisfies asymptotically the Poho\u{z}aev
type identity:
\[
P(u_n)\to 0,
\]
where $P(u) = 0$ is the Poho\u{z}aev identity related to the system \eqref{mainequation3} and $c$ is a
Mountain-pass energy given in \eqref{Mountainpass1} below.  The approach used in \cite{Jeanjean}, which deals with a local functional needs to be improved because of the effect of the nonlocal term.
\end{remark}

The paper is organized as follows. In Section 2, we provide several lemmas, which are crucial in proving
our main results. In Section 3, firstly the proof of Theorem \ref{maintheorem2} is obtained, then we complete the proof of Theorem \ref{maintheorem1}.
\\\\
\textbf{Notation.} Throughout this paper we shall denote by $C$ and $C_i$ ($i=1, 2,\cdots$) for various positive constants whose exact value may change from lines to lines but are not essential to the analysis of problem. $L^p(\R^3)$ $(1\leq p\leq+\infty)$ is the usual Lebesgue space with the standard norm $|u|_p$. For any Lebesgue measurable set $E\subset \R^3$, $|E|$ means the Lebesgue measure of the set $E$.
We use $``\to"$ and $``\rightharpoonup"$ to denote the strong and weak convergence in the related function space, respectively. The symbol $``\hookrightarrow"$ means a function space is continuously imbedding into another function space.
For any $\rho>0$ and any $x\in \R^3$, $B_\rho(x)$ denotes the ball of radius $\rho$ centered at $x$, that is, $B_\rho(x):=\{y\in \R^3:|y-x|<\rho\}$.

Let $(X,\|\cdot\|)$ be a Banach space with its dual space $(X^{-1},\|\cdot\|_{*})$, and $\Psi$ be its functional on $X$. The Palais-Smale sequence at level $c\in\R$ ($(PS)_c$ sequence in short) corresponding to $\Psi$ assumes that $\Psi(x_n)\to c$ and $\Psi^{\prime}(x_n)\to 0$ as $n\to\infty$, where $\{x_n\}\subset X$. If for any $(PS)_c$ sequence $\{x_n\}$ in $X$, there exists a subsequence $\{x_{n_{k}}\}$ such that $x_{n_{k}}\to x_0$ in $X$ for some $x_0\in X$, then we say that the functional $\Psi$ satisfies the so-called $(PS)_c$ condition.

\section{Variational settings and preliminaries}
 In this section, we will give some lemmas which are useful for the main results. To solve the system \eqref{mainequation1}, we introduce some function spaces. Throughout the paper, we consider the Hilbert space $H^1(\R^3)$ with the inner product and the norm as follows
$$
(u,v)=\int_{\R^3}\nabla u\nabla v+V(x)uvdx\ \  \text{and} \ \  \|u\|=\bigg(\int_{\R^3}|\nabla u|^2+ V(x)u^2dx\bigg)^{\frac{1}{2}},  \ \   \forall u,v\in H^1(\R^3)
$$
which are equivalent to the usual inner product and the norm in $H^1(\R^3)$ because of the assumption $(V_2)$.
The space
\[
D^{1,2}(\R^3)=\big\{u\in L^6(\R^3): \nabla u \in L^2(\R^3)\big\}
\]
equips with its usual inner product and norm
$$
(u,v)_{D^{1,2}(\R^3)}=\int_{\R^3}\nabla u\nabla vdx\ \ \text{and} \ \
\|u\|_{D^{1,2}(\R^3)}=\bigg(\int_{\R^3}|\nabla u|^2dx\bigg)^{\frac{1}{2}},
$$
respectively. The positive constant $S$ denotes the best Sobolev constant:
\begin{equation}\label{Sobolev}
  S:=\inf_{u\in D^{1,2}(\R^3)\setminus \{0\}}\frac{\int_{\R^3}|\nabla u|^2dx}{\big(\int_{\R^3}|u|^6dx\big)^{\frac{1}{3}}}.
\end{equation}

In the following, one can use the Lax-Milgram theorem, for every $u\in H^1(\R^3)$, there exists a unique $\phi_{u}\in D^{1,2}(\R^3)$ such that
\begin{equation}\label{Lax}
  -\Delta \phi= |u|^5
\end{equation}
 and $\phi_{u}$ can be written as
\begin{equation}\label{Riesz}
  \phi_{u}(x)=\frac{1}{4\pi}\int_{\R^3}\frac{ |u(y)|^5}{|x-y|}dy.
\end{equation}
Substituting \eqref{Riesz} into \eqref{mainequation1}, we get a single elliptic equation with a nonlocal term:
\[
-\Delta u+V(x)u- \phi_u|u|^3u=f(u)
\]
whose corresponding functional $I_{V }:H^1(\R^3)\to\R$ is defined by
\begin{equation}\label{functional}
  I_{V }(u)=\frac{1}{2}\|u\|^2-\frac{1}{10}\int_{\R^3} \phi_u|u|^5dx-\int_{\R^3}F(u)dx.
\end{equation}
We mention here that the idea of this reduction method was proposed by Benci and Fortunato \cite{Benci1} and it is a basic strategy for studying Schr\"{o}dinger-Poisson system today.

For simplicity, the conditions $(f_1)-(f_3)$ always hold true though out this paper and we don't assume them any longer unless specially needed. Thus for any $\epsilon>0$, there exists $C_\epsilon>0$ such that
\begin{equation}\label{growth}
  |uf(u)|\leq \epsilon|u|^2+C_\epsilon |u|^6\ \ \text{and}\ \ |F(u)|\leq \epsilon|u|^2+C_\epsilon |u|^6.
\end{equation}
To know more about the solution $\phi$ of the Poisson equation in \eqref{mainequation1} which can leads to a critical nonlocal term, we have the following key lemma:
\begin{lemma}\label{phi}
For every $u\in H^1(\R^3)$, we have the following conclusions:
\begin{enumerate}[$(1)$]
  \item $\phi_u(x)\geq 0$ for every $x\in\R^3$;
  \end{enumerate}
  \begin{enumerate}[$(2)$]
  \item $\|\phi_u\|_{D^{1,2}(\R^3)}^2=\int_{\R^3} \phi_u|u|^5dx$;
  \end{enumerate}
  \begin{enumerate}[$(3)$]
  \item for any $t>0$, $\phi_{tu}=t^5\phi_{u}$;
  \end{enumerate}
  \begin{enumerate}[$(4)$]
  \item if $u_n\rightharpoonup u$ in $H^1(\R^3)$, then $\phi_{u_n}\rightharpoonup \phi_u$ in $D^{1,2}(\R^3)$.
\end{enumerate}
\end{lemma}
\begin{proof}
As a direct consequence of \eqref{Lax} and \eqref{Riesz}, one can derive $(1)$, $(2)$ and $(3)$ at once.

For any $v\in D^{1,2}(\R^3)$, then $v\in L^6(\R^3)$ by \eqref{Sobolev}. Since $u_n\rightharpoonup u$ in $H^1(\R^3)$, then $u_n\rightharpoonup u$ in $L^6(\R^3)$ and $|u_n|^5\rightharpoonup |u|^5$ in $L^\frac{6}{5}(\R^3)$. Therefore
\[
(\phi_{u_n},v)_{D^{1,2}(\R^3)}=\int_{\R^3} |u_n|^5vdx\to\int_{\R^3} |u|^5vdx=(\phi_{u},v)_{D^{1,2}(\R^3)}
\]
which implies that $(4)$ is true.
\end{proof}
Furthermore, by $(2)$ of Lemma \ref{phi}, H\"{o}lder's inequality and \eqref{Sobolev}, one has
\[
\|\phi_{u}\|^2_{D^{1,2}}=\int_{\R^3} \phi_{u}|u|^5dx\leq  |\phi_{u}|_6|u|^5_6\leq  S^{-\frac{1}{2}}\|\phi_{u}\|_{D^{1,2}}|u|^5_6
\]
which implies that
\begin{equation}\label{Sobolev0}
  \|\phi_{u}\|_{D^{1,2}}\leq  S^{-\frac{1}{2}}|u|^5_6
\end{equation}
 and
\begin{equation}\label{Sobolev2}
  \int_{\R^3} \phi_{u}|u|^5dx\leq  S^{-6}\|u\|^{10}.
\end{equation}

From \eqref{growth} and \eqref{Sobolev2} we have that the functional $I_{V }$ given by \eqref{functional} is well-defined on $H^1(\R^3)$ and is of $C^1(H^1(\R^3),\R)$ class (see \cite{Willem}), and for any $v\in H^1(\R^3)$ one has
\[
\langle I^{\prime}_{V }(u),v\rangle=\int_{\R^3}\nabla u\nabla v+V(x)uvdx-\int_{\R^3} \phi_u|u|^3uvdx- \int_{\R^3}f(u)vdx.
\]
It is standard to verify that a critical point $u\in H^1(\R^3)$ of the functional $I_{V }$ corresponds to a
weak solution $(u, \phi_u)\in H^1(\R^3)\times D^{1,2}(\R^3)$ of \eqref{mainequation1}. In other words, if we can seek a critical point of the functional $I_{V }$, then the system \eqref{mainequation1} is solvable. In the following, we call $(u, \phi_u)$ is a positive solution of \eqref{mainequation1} if $u$ is a positive critical of the functional $I_{V }$. And $(u, \phi_u)$ is a least energy solution of \eqref{mainequation1} if the critical point $u$ of the functional $I_{V }$ verifies
\[
I_{V }(u)=\min_{v\in \mathcal{S}_{V }}I_{V }(v),
\]
where $\mathcal{S}_{V }:=\big\{u\in H^1(\R^3)\backslash\{0\}:I_{V }^{\prime}(u)=0\big\}$.

Motivated by the well-known Br\'{e}zis-Lieb lemma \cite{Brezis2}, we have the following important lemma to prove the convergence of Schr\"{o}dinger-Poisson system \eqref{mainequation1} involving a critical nonlocal term.
\begin{lemma}\label{weaklemma}
Let $r \geq 1$ and $\Omega$ be an open subset of $\R^N$. Suppose that $u_n\rightharpoonup u$
in $L^r(\Omega)$, and $u_n\to u$ $a.e.$ in $\Omega$ as $n\to\infty$, then
\[
|u_n|^p-|u_n-u|^p-|u|^p\to 0\ \ \text{in}\ \ L^{\frac{r}{p}}(\Omega)
\]
as $n\to\infty$ for any $p\in[1, r]$.
\end{lemma}

\begin{proof}
The proof is standard, so we omit it and the reader can refer in \cite[Lemma 2.2]{Li} for the detail proof.
\end{proof}

\begin{lemma}
If $u_n\rightharpoonup u$ in $H^1(\R^3)$, then going to a subsequence if necessary, we derive
\begin{equation}\label{weak1}
  |u_n|^5-|u_n-u|^5-|u|^5\to 0\ \ \text{in}\ \ L^{\frac{6}{5}}(\R^3),
\end{equation}
\begin{equation}\label{weak2}
  \phi_{u_n}-\phi_{u_n-u}-\phi_{u}\to 0\ \ \text{in}\ \ D^{1,2}(\R^3),
 \end{equation}
\begin{equation}\label{weak3}
\int_{\R^3} \phi_{u_n}|u_n|^5dx-\int_{\R^3} \phi_{u_n-u}|u_n-u|^5dx-\int_{\R^3} \phi_{u}|u|^5dx\to 0,
\end{equation}
and
\begin{equation}\label{weak4}
  \int_{\R^3} \phi_{u_n}|u_n|^3u_n\varphi dx-\int_{\R^3} \phi_{u}|u|^3u\varphi dx\to 0
\end{equation}
for any $\varphi \in C^\infty_0(\R^3)$.
\end{lemma}

\begin{proof}
Since $u_n\rightharpoonup u$ in $H^1(\R^3)$, then $u_n\rightharpoonup u$ in $L^6(\R^3)$. And $u_n\to u$ $a.e.$ in $\R^3$ because $u_n\to u$ in $L_{\text{loc}}^s(\R^3)$ with $1\leq s<6$ in the sense of a subsequence.
If we take $r=6$ and $p=5$ in Lemma \ref{weaklemma}, one has \eqref{weak1} immediately.

It follows from $(2)$ of Lemma \eqref{phi} and H\"{o}lder's inequality that
\begin{eqnarray*}
 \big|(\phi_{u_n}-\phi_{u_n-u}-\phi_{u},w)_{D^{1,2}(\R^3)}\big| &=& \big| \int_{\R^3}\nabla(\phi_{u_n}-\phi_{u_n-u}-\phi_{u})\nabla wdx \big|\\
   &=&  \big|\int_{\R^3} (|u_n|^5-|u_n-u|^5-|u|^5) wdx \big|\\
    &\leq&  |w|_6 \big|(|u_n|^5-|u_n-u|^5-|u|^5)\big|_{\frac{6}{5}},
\end{eqnarray*}
which implies that
\begin{eqnarray*}
 && \sup_{w\in D^{1,2}(\R^3),\ \ \|w\|_{D^{1,2}(\R^3)}=1}\big|(\phi_{u_n}-\phi_{u_n-u}-\phi_{u},w)_{D^{1,2}(\R^3)}\big| \\
   &\leq&  \big|(|u_n|^5-|u_n-u|^5-|u|^5)\big|_{\frac{6}{5}}\stackrel{\mathrm{\eqref{weak1}}}{\to}0,
  \end{eqnarray*}
hence \eqref{weak2} holds.

Using \eqref{weak2}, one has $\phi_{u_n}-\phi_{u_n-u}-\phi_{u}\to 0$ in $L^6(\R^3)$. Since $\{u_n\}$ is bounded in $L^6(\R^3)$, then by H\"{o}lder's inequality,
\begin{eqnarray*}
|A_1| &:=& \big|\int_{\R^3} \big(\phi_{u_n}-\phi_{u_n-u}-\phi_{u}\big)|u_n|^5dx\big| \\
   &\leq&  |u_n|_6^5 \big|\phi_{u_n}-\phi_{u_n-u}-\phi_{u}\big|_6\to 0.
\end{eqnarray*}
Similarly, one can deduce that
\[
A_2:=\int_{\R^3} \big(\phi_{u_n}-\phi_{u_n-u}-\phi_{u}\big)|u|^5dx\to 0.
\]
In view of \eqref{Sobolev0}, $\{\phi_{u_n-u}\}$ is bounded in $L^6(\R^3)$, then using H\"{o}lder's inequality again,
 \begin{eqnarray*}
 |A_3|  &:=&\big|\int_{\R^3} \phi_{u_n-u}\big(|u_n|^5-|u_n-u|^5-|u|^5\big)dx\big|  \\
   &\leq&  |\phi_{u_n-u}|_6 \big|(|u_n|^5-|u_n-u|^5-|u|^5)\big|_{\frac{6}{5}} \stackrel{\mathrm{\eqref{weak1}}}{\to}0.
\end{eqnarray*}
As $u_n\rightharpoonup u$ in $H^1(\R^3)$, then one has $\phi_{u_n}\rightharpoonup \phi_{u}$ in $D^{1,2}(\R^3)$ by $(4)$ of Lemma \ref{phi} and thus $\phi_{u_n}\rightharpoonup \phi_{u}$ in $L^{6}(\R^3)$. Clearly, $ |u|^5\in L^{\frac{6}{5}}(\R^3)$, thus
\[
A_4:=\int_{\R^3} (\phi_{u_n}-\phi_{u})|u|^5dx\to 0.
\]
By $u_n\rightharpoonup u$ in $H^1(\R^3)$, one has $|u_n|^5\rightharpoonup |u|^5$ in $L^\frac{6}{5}(\R^3)$. Since $ \phi_u\in L^{6}(\R^3)$, then
\[
A_5:=\int_{\R^3} \phi_{u}(|u_n|^5-|u|^5)dx\to 0.
\]
Consequently,
\begin{eqnarray*}
 && \int_{\R^3} \phi_{u_n}|u_n|^5dx-\int_{\R^3} \phi_{u_n-u}|u_n-u|^5dx-\int_{\R^3} \phi_{u}|u|^5dx  \\
    &=& A_1-A_2+A_3+ A_4+A_5\\
    &\to& 0,
\end{eqnarray*}
which shows that \eqref{weak3} is true.

Since $u_n\rightharpoonup u$ in $H^1(\R^3)$, then one can deduce again that $\phi_{u_n}\rightharpoonup \phi_u$ in $L^6(\R^3)$. By $ |u|^3u\varphi\in L^{\frac{6}{5}}(\R^3)$ because $\varphi\in C^\infty_0(\R^3)$, one has
\[
 \int_{\R^3} \phi_{u_n}|u|^3u \varphi dx-\int_{\R^3} \phi_{u}|u|^3u\varphi dx\to 0.
\]
On the other hand, by means of H\"{o}lder's inequality and $\{\phi_{u_n}\}$ is bounded in $L^6(\R^3)$,
\begin{eqnarray*}
&&\bigg|\int_{\R^3} \phi_{u_n}|u_n|^3u_n\varphi dx-\int_{\R^3} \phi_{u_n}|u|^3u\varphi dx\bigg|\\
   &\leq &  \int_{\supp \varphi}|\phi_{u_n}||\varphi|\big||u_n|^3u_n-|u|^3u\big|dx                                                             \\
   &\leq &  |\phi_{u_n}|_{6}|\varphi|_{\infty}\bigg(\int_{\supp \varphi}\big||u_n|^3u_n-|u|^3u\big|^{\frac{6}{5}}dx\bigg)^{\frac{5}{6}}\to 0,
  \end{eqnarray*}
where we have used $u_n\to u$ in $L_{\text{loc}}^s(\R^3)$ with $1\leq s<6$ in the sense of a subsequence.
As a consequence of the above two facts, one has
\[
\int_{\R^3} \phi_{u_n}|u_n|^3 u_n\varphi dx-\int_{\R^3} \phi_{u}|u|^3 u\varphi dx
   \to 0.
\]
The proof is complete.
\end{proof}

In this paper, the Poho\u{z}aev identity (see \cite{Berestycki,dAprile2} for example) will paly an vital role.
Here we prove a more generalized Poho\u{z}aev identity which can be seen as an partial extension to that of in \cite{Li}.

\begin{lemma}\label{Pohozaev}
(Poho\u{z}aev identity) Assume $V$ and $f$ satisfy $(V_1)-(V_2)$ and $(f_1)-(f_3)$, respectively. Let $(u,\phi_u)\in H^1(\R^3)\times D^{1,2}(\R^3)$ be a weak solution for \eqref{mainequation1}, then we have the following Poho\u{z}aev identity:
\begin{align}\label{Pohozaev1}
\nonumber   P_{V}(u)&\triangleq \frac{1}{2}\int_{\R^3}|\nabla u|^2dx+\frac{3}{2}\int_{\R^3}V(x)|u|^2dx+\frac{1}{2}\int_{\R^3}(x,\nabla V)|u|^2dx-\frac{1}{2}\int_{\R^3} \phi_u|u|^5dx  \\
&\ \ \ \  -3\int_{\R^3}F(u)dx\equiv 0.
\end{align}
In particular if $V(x) \equiv 1$, the above Poho\u{z}aev identity can be rewritten as follows:
\begin{equation}\label{Pohozaev2}
P(u)\triangleq\frac{1}{2}\int_{\R^3}|\nabla u|^2dx+\frac{3}{2}\int_{\R^3}|u|^2dx-\frac{1}{2}\int_{\R^3} \phi_u|u|^5dx -3\int_{\R^3}F(u)dx\equiv 0.
\end{equation}
\end{lemma}
\begin{proof}
For any $R>0$, by multiplying $x\cdot\nabla u$ and $x\cdot\nabla \phi$ on the both sides for the two equations in \eqref{mainequation1} respectively and then integrating on $B_R$, one has
\begin{equation}\label{Pohozaev3}
\int_{B_R}-\Delta u(x,\nabla u)dx=-\frac{1}{2}\int_{B_R}|\nabla u|^2dx-\frac{1}{R}\int_{\partial B_R}|(x,\nabla u)|^2dx+\frac{R}{2}\int_{\partial B_R}|\nabla u|^2dx,
 \end{equation}
\begin{equation}\label{Pohozaev4}
 \int_{B_R} V(x)u(x,\nabla u)dx=-\frac{3}{2}\int_{B_R}V(x)|u|^2dx-\frac{1}{2}\int_{B_R}|u|^2(x,\nabla V)dx+\frac{R}{2}\int_{\partial B_R}|u|^2dx,
\end{equation}
\begin{equation}\label{Pohozaev5}
   \int_{B_R} \phi|u|^3 u(x,\nabla u)dx=-\frac{1}{5}\int_{B_R} |u|^5(x,\nabla \phi)dx-
  -\frac{3}{5}\int_{B_R} \phi|u|^5dx
+\frac{R}{5}\int_{\partial B_R} \phi|u|^5dx,
\end{equation}
\begin{equation}\label{Pohozaev55}
 \int_{B_R}f(u)(x,\nabla u)dx=-3\int_{B_R}F(u)dx+R\int_{\partial B_R}F(u)dx,
\end{equation}
 \begin{equation}\label{Pohozaev6}
 \int_{B_R}-\Delta \phi(x,\nabla \phi)dx=-\frac{1}{2}\int_{B_R}|\nabla \phi|^2dx-\frac{1}{R}\int_{\partial B_R}|(x,\nabla \phi)|^2dx+\frac{R}{2}\int_{\partial B_R}|\nabla \phi|^2dx.
\end{equation}
and
\begin{equation}\label{Pohozaev7}
 \int_{B_R} |u|^5(x,\nabla \phi)dx=\int_{B_R}-\Delta \phi(x,\nabla \phi)dx.
\end{equation}
It follows from \eqref{Pohozaev3}-\eqref{Pohozaev55} and letting $R\to+\infty$ that
\begin{align}\label{Pohozaev8}
\nonumber&\frac{1}{2}\int_{\R^3}|\nabla u|^2dx+\frac{3}{2}\int_{\R^3}V(x)|u|^2dx+\frac{1}{2}\int_{\R^3}|u|^2(x,\nabla V)dx-
\frac{1}{5}\int_{\R^3} |u|^5(x,\nabla \phi)dx\\
    &   -\frac{3}{5}\int_{\R^3} \phi|u|^5dx
 =3\int_{\R^3}F(u)dx.
 \end{align}
Using \eqref{Pohozaev6}-\eqref{Pohozaev8} and letting $R\to+\infty$ again, one can deduce that
\begin{align*}
 &\frac{1}{2}\int_{\R^3}|\nabla u|^2dx+\frac{3}{2}\int_{\R^3}V(x)|u|^2dx+\frac{1}{2}\int_{\R^3}|u|^2(x,\nabla V)dx+
\frac{1}{10}\int_{\R^3}|\nabla \phi|^2dx\\
    &  -\frac{3}{5}\int_{\R^3} \phi|u|^5dx
 =3\int_{\R^3}F(u)dx.
\end{align*}
Recalling that
\[
\int_{\R^3}|\nabla \phi|^2dx=\int_{\R^3} \phi|u|^5dx,
\]
then by means of the above two formulas, we obtain \eqref{Pohozaev1}. Hence \eqref{Pohozaev2} is trivial.
\end{proof}

\begin{lemma}\label{Mountainpass}
The functional $I(u)$ corresponding to \eqref{mainequation3} satisfies the Mountain-pass geometry around $0\in H^1(\R^3)$, that is,
\begin{enumerate}[$(i)$]
  \item there exist $\alpha_1,\rho_1>0$ such that $I(u)\geq \alpha_1>0$ when $\|u\|=\rho_1$;
  \item there exists $e\in H^1(\R^3)$ with $\|e\|>\rho_1$ such that $I(e)<0$.
\end{enumerate}
\end{lemma}

\begin{proof}
 $(i)$ By the definition of $I$, \eqref{growth}, \eqref{Sobolev2} and Sobolev theorem, one has
 \begin{eqnarray*}
I(u)&=& \frac{1}{2}\int_{\R^3}|\nabla u|^2dx+\frac{1}{2}\int_{\R^3}|u|^2dx-\frac{1}{10}\int_{\R^3} \phi_u |u|^5 dx -\int_{\R^3}F(u)dx\\
&\geq&\frac{1}{2}\|u\|^2-\frac{1}{10} S^{-6}\|u\|^{10}-C\epsilon\|u\|^{2}-CC_\epsilon\|u\|^{6},
\end{eqnarray*}
Therefore letting $\epsilon=\frac{1}{4C}>0$, then there exist constants
$\alpha_1,\rho_1>0$ such that $I(u)\geq \alpha_1$ whenever $\|u\|=\rho_1$.

$(ii)$ For any $u\in H^1(\R^3)\backslash\{0\}$, it follows from the assumption $(f_1)$ that
 \[
 I(tu)\leq\frac{t^2}{2}\|u\|^2-\frac{t^{10}}{10}\int_{\R^3} \phi_u |u|^5 dx\to-\infty
 \]
 as $t\to+\infty$. Hence choosing $e=t_0u$ for some $u\in H^1(\R^3)\backslash\{0\}$ with
$t_0$ large enough, then one has $\|e\|>\rho_1$ and $I(e)<0$.
\end{proof}

As a consequence of Lemma \ref{Mountainpass}, we can find a $(PS)$ sequence of the functional $I(u)$ at the level
\begin{equation}\label{Mountainpass1}
 c:=\inf_{\eta\in \Gamma}\max_{t\in [0,1]}I(\eta(t))>0,
\end{equation}
where the set of paths is defined as
\begin{equation}\label{Mountainpass2}
  \Gamma:=\big\{\eta\in C([0,1],H^1(\R^3)):\eta(0)=0, I(\eta(1))<0\big\}.
\end{equation}

\begin{remark}
It is easy to see that
\begin{equation}\label{compare}
  c\leq \inf_{u\in H^1(\R^3)\backslash\{0\}}\max_{t\geq 0}I(tu).
\end{equation}
Indeed, for any $u\in H^1(\R^3)\backslash\{0\}$, similar to Lemma \ref{Mountainpass} $(ii)$ there exists a sufficiently large $t_0>0$ such that $I(t_0u)<0$. Let us choose $\gamma_0(t)=tt_0u$, therefore $\gamma_0\in C([0,1],H^1(\R^3))$ and moreover $\gamma_0\in \Gamma$, thus
\[
c\leq \max_{t\in [0,1]}I(\gamma_0(t))=\max_{t\in [0,1]}I(tt_0u)=\max_{t\in [0,t_0]}I(tu)\leq \max_{t\geq 0}I(tu).
\]
Since $u\in H^1(\R^3)\backslash\{0\}$ in arbitrary, then \eqref{compare} holds.
\end{remark}

Even though a $(PS)_c$ sequence has been established, we can never easily affirm that whether this $(PS)_c$ sequence is bounded because the Ambrosetti-Rabinowitz type condition or monotonicity assumption on the nonlinearity is not satisfied. To overcome it, we will construct a $(PS)$ sequence $\{u_n\}$ for $I$ at the level $c$ that satisfies $P(u_n)\to 0$ as $n \to \infty $, here $P(u)$ is given by \eqref{Pohozaev2}.

\begin{proposition}\label{getPScsequence}
There exists a sequence $\{u_n\}\subset H^1(\R^3)$ such that
\begin{equation}\label{PScsequence}
I(u_n) \to c,~I^{\prime}(u_n)\to 0,~P(u_n)\to 0~\text{as}~n \to \infty .
\end{equation}
\end{proposition}

\begin{proof}
Following the idea introduced by L. Jeanjean \cite{Jeanjean},
we define a map $\Phi(\tau,v) :\R \times {H^1}(\R^3) \to H^1(\R^3)$ for $\tau\in\R$, $v \in H^1(\R^3)$ and $x \in \R^3$ by $$
\Phi(\tau,v) = v(e^{-\tau}x).
$$
For every $\tau\in\R$, $v\in H^1(\R^3)$, the functional $I\circ\Phi(\tau,v)\in C^1(\R \times {H^1}(\R^3),\R)$ is computed as
\[
I\circ \Phi= \frac{e^{\tau}}{2}\int_{\R^3}|\nabla v|^2dx+\frac{e^{3\tau}}{2}\int_{\R^3}|v|^2dx
-\frac{e^{5\tau}}{10}\int_{\R^3}\phi |v|^5dx
  -e^{3\tau}\int_{\R^3}F(v)dx.
\]
Similar to Lemma \ref{Mountainpass}, we can easily check that $I\circ\Phi(\tau,v)>0$ for all $(\tau,v)$ with $|\tau|>0$, $\|v\|>0$ small and $(I\circ\Phi)(0,u_0)< 0$, $i.e.$ $I\circ\Phi$ possesses the Mountain-pass geometry in $\R\times H^1(\R^3)$. Hence we can define the Mountain-pass level of $I\circ\Phi$:
\begin{equation}\label{Mountainpass3}
{\tilde c}: = \mathop {\inf }\limits_{\tilde \eta\in {{\tilde \Gamma }}} \sup_{\theta\in[0,1]}(I\circ\Phi)(\tilde \eta (\theta)),
\end{equation}
where the set of paths is defined as
\begin{equation}\label{Mountainpass4}
{\tilde\Gamma}:=\big\{{\tilde\eta\in C([0,1],\R\times H^1(\R^3)):\tilde\eta(0)=(0,0),~(I\circ\Phi)(\tilde\eta(1))<0} \big\}.
\end{equation}
It is easy to see that $\{0\}\times\Gamma\subset\tilde{\Gamma}$ and then $\tilde{c}\leq c$. On the other hand, $\Phi(\tilde{\Gamma})\subset \Gamma$ gives us that $\tilde{c}\geq c$. Therefore,
 the Mountain-pass levels of $I$ and $I\circ\Phi$ coincide, that is,
\begin{equation}\label{cc}
  c = \tilde c.
\end{equation}
Let $\gamma_n\subset \Gamma$ satisfy $\max_{\theta\in[0,1]}I(\gamma_n(\theta))\leq c+\frac{1}{n^2}$ and set $\widetilde{\gamma}_n(\theta)=(0,\gamma_n(\theta))$, then $\widetilde{\gamma}_n(\theta)\subset \widetilde{\Gamma}$ and by \eqref{cc}, we derive
\[
\max_{\theta\in[0,1]}(I\circ \Phi)(\widetilde{\gamma}_n(\theta))\leq \widetilde{c}+\frac{1}{n^2}.
\]
Hence by means of the Ekeland's variational principle \cite{Ekeland}, we can see that there exists a sequence $\{(\tau_n,v_n)\}_{n \in \mathbb{N}}$ in $\R\times H^1(\R^3)$ such that,
\begin{equation}\label{Mountainpass5}
(I\circ \Phi )(\tau_n,v_n) \to c,
\end{equation}
\begin{equation}\label{Mountainpass6}
(I\circ \Phi )^{\prime}({\tau_n},{v_n}) \to 0{\text{ in (}}\R \times {H^1}({\R^3}){)^{ - 1}},
\end{equation}
 and
\begin{equation}\label{Mountainpass66}
  \min_{\theta\in[0,1]}\|({\tau_n},{v_n})-\widetilde{\gamma}_n(\theta)\|_{\R \times H^1(\R^3)}\to 0
\end{equation}
as $n\to\infty$, where the norm of $(\tau,u)\in \R \times H^1(\R^3)$ is defined as usual as $\|(\tau,u)\|=(\tau^2+\|u\|^2)^{\frac{1}{2}}$.

Set $u_n=\Phi(\tau_n,v_n)$, then by \eqref{Mountainpass5} we deduce that
\[
I(u_n)\to c.
\]
For every $(h,w) \in \R\times H^1(\R^3)$, by \eqref{Mountainpass6} we have
 \begin{equation}\label{Mountainpass7}
\big\langle (I\circ\Phi)^{\prime}(\tau_n,v_n),(h,w)\big\rangle  = \big\langle I^{\prime}(\Phi(\tau_n,v_n)),\Phi (\tau_n,w)\big\rangle + P(\Phi(\tau_n,v_n))h\to 0.
\end{equation}
Taking $h=1$, $w=0$ in \eqref{Mountainpass7}, we get
\[
P(u_n)=P(\Phi(\tau_n,v_n)) \to 0.
\]
If we take $h=0$ in \eqref{Mountainpass7}, then we have
\[
\langle I^{\prime}(u_n),w \rangle=\big\langle (I\circ\Phi)^{\prime}(\theta_n,v_n),(\Phi(-\theta_n),0)\big\rangle,
\]
and combing with \eqref{Mountainpass66}, we derive
$$
I^{\prime}(u_n)\to 0.
$$
Hence, we have got a sequence $\{u_n\}_{n=1}^\infty\subset H^1(\R^3)$ that satisfies \eqref{PScsequence}.
\end{proof}

To find a least energy solution for problem \eqref{mainequation3}, we set
$$
m\triangleq\inf\big\{I(w):w\in H^1(\R^3)\backslash\{0\}  \ \ \text{and}\ \ P(w)=0\big\}.
$$

\begin{lemma}\label{compare}
$m\geq c$, where $c$ is given by \eqref{Mountainpass1}.
\end{lemma}

\begin{proof}
For any $w\in H^1(\R^3)\backslash\{0\}$ with $P(w) =0$, then letting $w_t\triangleq w(\frac{x}{t})$ if $t>0$, and $w_t\triangleq w(x)$ if $t=0$. Hence we have that
\begin{eqnarray*}
I(w_t)&=& I(w_t)-\frac{t^3}{3}P(w) \\
   &=&\bigg(\frac{t}{2}-\frac{t^3}{6}\bigg)\int_{\R^3}|\nabla w|^2dx+\bigg(\frac{t^3}{6}-\frac{t^5}{10}\bigg)\int_{\R^3} \phi_{w} |w|^5dx.
\end{eqnarray*}
Clearly, there exists a sufficiently large $T_0>0$ such that $I(w_{T_0}) <0$ and $I(w_t)$ achieves the strict global maximum at $t=1$, that is, $I(w)=\max_{t>0}I(w_t)$.

On the other hand, we set the path
$$
\eta_0(t)=
\left\{%
\begin{array}{ll}
   w(\frac{x}{tT_0}) , & t>0, \\
    0, & t=0,\\
\end{array}%
\right.
$$
thus $\eta_0(0)=0$, $I(\eta_0(1))=I(w_{T_0}) <0$ and then $\eta_0(t)\in \Gamma$, where $\Gamma$ is given by \eqref{Mountainpass2}. As a consequence of this fact, we can infer that
\[
I(w)=\max_{t\geq0}I(w_t)=\max_{t\geq0}I(w_{tT_0})=\max_{t\geq0}I\big(\eta_0(t)\big)\geq \max_{t\in[0,1]}I\big(\eta_0(t)\big)\stackrel{\mathrm{\eqref{Mountainpass1}}}{\geq} c.
\]
Since $w\in H^1(\R^3)\backslash\{0\}$ with $P(w) =0$ is arbitrary, we have that $m\geq c$.
\end{proof}

Because of the appearance of the critical nonlocal term, we have to estimate the Mountain-pass value given by \eqref{Mountainpass1} carefully.
To do it, we choose the extremal function
$$
U_{\epsilon }(x)=\frac{(3\epsilon^2)^{\frac{1}{4}}}{(\epsilon^2+|x |^2)^{\frac{1}{2}}}
$$
to solve $-\Delta u=u^5$ in $\R^3$. Let
$\varphi\in C_0^{\infty}(\R^3)$ be a cut-off function verifying that $0\leq \varphi(x)\leq 1$ for all $x\in \R^3$, $\supp \varphi\subset B_{2}(x_0)$, and $\varphi(x)\equiv 1$ on $B_1(x_0)$.
Set $v_{\epsilon }=\varphi U_{\epsilon }$,
then thanks to the asymptotic estimates from \cite{Brezis1}, we have
\begin{equation}\label{estimate1}
  |\nabla v_{\epsilon }|_2^2=S^{\frac{3}{2}}+O(\epsilon), \ \  |v_{\epsilon }|_6^2=S^{\frac{1}{2}}+O(\epsilon)
\end{equation}
and for all $s\in [2,6)$,
\begin{equation}\label{estimate2}
  |v_{\epsilon }|_s^s=
\left\{
  \begin{array}{ll}
   O(\epsilon^{\frac{s}{2}}), & \text{if}~s\in[2,3), \\
     O(\epsilon^{\frac{3}{2}}|\log \epsilon|), & \text{if}~s=3, \\
     O(\epsilon^{\frac{6-s}{2}}), &   \text{if}~s\in(3,6).
  \end{array}
\right.
\end{equation}

\begin{lemma}\label{estimate}
For $q\in(3,5)$ with any $\mu>0$, or $q\in(1,3]$ with sufficiently large $\mu>0$, the Mountain-pass value given by \eqref{Mountainpass1} satisfies
\[
   c<\frac{2}{5}S^{\frac{3}{2}},
\]
where $S$ is the best Sobolev constant given in \eqref{Sobolev}.
\end{lemma}

\begin{proof}
This lemma can be seen as a direct corollary of Lemma \ref{estimateVK} in Section 3 where a more general case is taken into account, so
we do not give its detail proof.
\end{proof}

Now we introduce the Vanishing lemma for Sobolev space as follows.
\begin{lemma}\label{Vanishing}
\
\big(see \cite{P. L. Lions1,P. L. Lions2}\big) Assume that $\{u_n\}$ is bounded in $H^1(\R^3)$ and satisfies
\[
\lim_{n\to\infty}\sup_{y\in\R^3}\int_{B_r(y)}|u_n|^2dx=0
\]
for some $r>0$. Then $u_n\to 0$ in $L^m(\R^3)$ for every $2<m<6$.
\end{lemma}
\section{Proof of main results}

In this section, we will show the proofs for Theorem \ref{maintheorem1} and Theorem \ref{maintheorem2} in detail. Before we prove Theorem \ref{maintheorem1}, we obtain Theorem \ref{maintheorem2} for convenience.
\vskip3mm
\begin{proof}[\textbf{Proof of Theorem \ref{maintheorem2}}]
Let $\{u_n\}$ be the sequence given in \eqref{PScsequence} and $c$ be the Mountain-pass value for the functional $I$, respectively. It is easy to show that $\{u_n\}$ is bounded is $H^1(\R^3)$. Indeed, the following computations are elementary, that is,
\begin{eqnarray*}
c+o(1)&\stackrel{\mathrm{\eqref{PScsequence}}}{=} & I(u_n)-\frac{1}{3}P(u_n) \\
   &=&\frac{1}{3}\int_{\R^3}|\nabla u_n |^2dx+\frac{1}{15}\int_{\R^3} \phi_{u_n} |u_n|^5dx\\
 &\geq&\frac{1}{3}\int_{\R^3}|\nabla u_n |^2dx\stackrel{\mathrm{\eqref{Sobolev}}}{\geq} \frac{S}{3}\bigg(\int_{\R^3}|u_n|^6dx\bigg)^{\frac{1}{3}}
\end{eqnarray*}
and in view of \eqref{growth}, \eqref{Sobolev0}, \eqref{PScsequence}, one has
\begin{eqnarray*}
\frac{1}{2}\int_{\R^3} |u_n|^2dx&\leq&\frac{1}{2}\int_{\R^3}|\nabla u_n|^2+|u_n|^2dx \\
   &{=}&c+\frac{1}{10}\int_{\R^3} \phi_{u_n} |u_n|^5dx+\int_{\R^3} F(u_n)dx+o(1)\\
 &{\leq}&c+\frac{1}{10S}|u_n|_6^{10}
 +\epsilon\int_{\R^3}u_n^2dx+C_\epsilon\int_{\R^3}|u_n|^6dx+o(1).
\end{eqnarray*}
Let $\epsilon=\frac{1}{4}>0$, thus there exists $C>0$ independent on $n\in \mathbb{N}$ such that $|u_n|_2\leq C<+\infty$ and then we can infer that $\{u_n\}$ is bounded is $H^1(\R^3)$. Therefore there exists $u\in H^1(\R^3)$ such that $u_n\rightharpoonup u$ in $H^1(\R^3)$ after passing to a subsequence. To end the proof, we will split it into several steps:
\vskip0.3cm
\underline{\textbf{Step 1:}} \ \ \ \   $u\not\equiv 0$.
\vskip0.3cm
\noindent In fact, we will argue it indirectly and just suppose
\begin{equation}\label{proof1}
 \lim_{n\to\infty}\sup_{y\in\R^3}\int_{B_r(y)}|u_n|^2dx=0,
\end{equation}
then from Lemma \ref{Vanishing} we derive that $u_n\to 0$ in $L^m(\R^3)$ for any $m\in(2,6)$. Thus by using Strass's compactness lemma in \cite{Berestycki} and the assumptions $(f_1)-(f_2)$, one has
\[
\int_{\R^3}F(u_n)dx\to 0\ \ \text{and}\ \ \int_{\R^3}u_nf(u_n)dx\to 0.
\]
Recalling \eqref{PScsequence} and the above formulas, we have the following results at once:
\[
I(u_n)=\frac{1}{2}\|u_n\|^2-\frac{1}{10}\int_{\R^3} \phi_{u_n}|u_n|^5dx=c+o(1)
\]
and
\[
\langle I^{\prime}(u_n),u_n\rangle=\|u_n\|^2-\int_{\R^3} \phi_{u_n}|u_n|^5dx=o(1).
\]
Hence without loss of generality, we may assume
\[
\lim_{n\to\infty}\|u_n\|^2=\lim_{n\to\infty}\int_{\R^3} \phi_{u_n}|u_n|^5dx=l,\ \  \text{and}\ \ c=\frac{2}{5}l.
\]
On the other hand, by \eqref{Sobolev2} we can deduce that
\[
\int_{\R^3} \phi_{u_n}|u_n|^5dx\leq  S^{-6}\|u_n\|^{10}
\]
which implies that $l\leq  S^{-6}l^5$. Hence either $l=0$ or $l\geq  S^{\frac{3}{2}}$. But $l=0$ yields that $c=0$ which is a contradiction to \eqref{Mountainpass1}, so $l\geq  S^{\frac{3}{2}}$. However
\[
c=\frac{2}{5}l\geq \frac{2}{5} S^{\frac{3}{2}}
\]
which also yields a contradiction to Lemma \ref{estimate}. Therefore \eqref{proof1} can never hold and
then there exist $r,\eta>0$ such that
\[
\lim_{n\to\infty}\int_{B_r(y_n)}|u_n|^2dx\geq \eta>0,
\]
here we may assume $y_n\in \mathbb{Z}^3$ by taking a larger $r$ if necessary. Since $I$ is invariant under translations by \eqref{Riesz}, we may assume that $\{y_n\}$ is bounded in $\mathbb{Z}^3$. Thus, passing to a subsequence we have $u_n\rightharpoonup u\neq 0$ in $H^1(\R^3)$.
\vskip0.3cm
\underline{\textbf{Step 2:}} \ \ \ \   $I^{\prime}(u)=0$.
\vskip0.3cm
\noindent To see this, since $C_0^{\infty}(\R^3)$ is dense in $H^1(\R^3)$, then it suffices to show
\[
\langle I^{\prime}(u),\varphi\rangle=0\ \ \text{for any}\ \ \varphi\in C_0^{\infty}(\R^3).
\]
Indeed $\{u_n\}$ is bounded in $H^1(\R^3)$, then there is $M\in(0,+\infty)$ such that
$$
\max\bigg\{|u_n|_2,|u|_2,|u_n|_6^5,|u|_6^5,|\varphi|_2,|\varphi|_{12}\bigg\}\leq M<+\infty.
$$
Hence for any $\sigma>0$, there exists $\delta=\delta(\sigma)\triangleq(\frac{\sigma}{2M^2C_\sigma})^{12}>0$ such that for any measurable set $E\subset \R^3$ with $|E|<\delta$ one has
\begin{eqnarray*}
  \int_{E} |f(u_n)\varphi|dx  &\stackrel{\mathrm{\eqref{growth}}}{\leq}& \frac{\sigma}{2M^2}\int_{E} | u_n||\varphi|dx+C_\sigma\int_{E} |u_n|^5|\varphi|dx  \\
   &\leq&   \frac{\sigma}{2M^2}|u_n|_2|\varphi|_2+C_\sigma|E|^{\frac{1}{12}}|u_n|_6^5|\varphi|_{12}<\sigma
\end{eqnarray*}
which reveals that $\{f(u_n)\varphi\}$ is uniformly integrable in $\R^3$. As a consequence of the Vitali's Dominated Convergence Theorem, we can deduce that
 \begin{equation}\label{weak6}
   \int_{\R^3}  f(u_n)\varphi dx =  \int_{\supp \varphi}f(u_n)\varphi dx\to   \int_{\supp \varphi}f(u)\varphi dx=\int_{\R^3}  f(u)\varphi dx.
 \end{equation}
 On the other hand, it follows from \eqref{weak4} and \eqref{weak6} that
\begin{eqnarray*}
 \langle I^{\prime}(u_n),\varphi\rangle&=& \int_{\R^3}\big[\nabla u_n \nabla\varphi +u_n\varphi\big] dx- \int_{\R^3} \phi_{u_n}|u_n|^3 u_n\varphi dx
 -\int_{\R^3}f(u_n)\varphi dx  \\
    &\to&\int_{\R^3}\big[\nabla u \nabla\varphi +u\varphi\big] dx- \int_{\R^3}\phi_{u}|u|^3 u\varphi dx-\int_{\R^3}f(u)\varphi dx  \\
   &=&  \langle I^{\prime}(u),\varphi\rangle
\end{eqnarray*}
which implies that
\[
\langle I^{\prime}(u),\varphi\rangle=0
\]
by \eqref{PScsequence} for any $\varphi\in C_0^{\infty}(\R^3)$.
\vskip0.3cm
\underline{\textbf{Step 3:}} \ \ \ \   $I(u)=m$ \ \ and \ \ $u(x)>0$\ \  in\ \  $\R^3$.
\vskip0.3cm
\noindent Using Lemma \ref{Pohozaev} and the Step 2 that $P(u)=0$, so by Fatou's lemma and Lemma \ref{compare}
\begin{eqnarray*}
c&\leq&m \ \ \leq\ \  I(u)=I(u)-\frac{1}{3}P(u) \\
   &=&\frac{1}{3}\int_{\R^3}|\nabla u |^2dx+\frac{1}{15}\int_{\R^3} \phi_{u} |u|^5dx\\
&\leq&\mathop{\lim\inf}_{n\to\infty}\bigg\{\frac{1}{3}\int_{\R^3}|\nabla u_n |^2dx+\frac{1}{15}\int_{\R^3} \phi_{u_n} |u_n|^5dx\bigg\} \\
&=&\mathop{\lim\inf}_{n\to\infty}\big\{I(u_n)-\frac{1}{3}P(u_n)\big\}\stackrel{\mathrm{\eqref{PScsequence}}}{=} c
\end{eqnarray*}
which reveals that $I(u)=m$.
The remainder is to show that $u(x)>0$ in $\R^3$. In fact, it is obvious that $|u|$ is also
a least energy solution of \eqref{mainequation3} since the functional $I$ is
symmetric, hence we may assume that such a least energy solution does not change sign,
$i.e.$ $u \geq0$. By means of the strong maximum principle and standard arguments, see $e.g.$
\cite{Alves,Benedetto,G. Li0,Moser,Trudinger}, we obtain that $u(x) > 0$ for all $x\in\R^3$. Thus, $(u,\phi_u)$ is a positive least energy solution of \eqref{mainequation3} and the proof is complete.
\end{proof}

Now we begin to deal with the case that $V(x)$ is not a constant, however the method to prove Theorem \ref{maintheorem2} can not be applied because of the effect on $V(x)$. In order to get a bounded $(PS)_c$ sequence for the functional $I_{V}$, we make use of the monotone method introduced by L. Jeanjean \cite{Jeanjean2}.

\begin{proposition}\label{proposition}
(See \cite[Theorem 1.1 and Lemma 2.3]{Jeanjean2}) Let $(X,\|\cdot\|)$ be a Banach space and $T\subset R^+$ be an interval, consider a family of $C^1$ functionals on $X$ of the form
$$
\Phi_{\lambda}(u)=A(u)-\lambda B(u),\ \  \forall \lambda\in T,
$$
with $B(u)\geq 0$ and either $A(u)\to +\infty$ or $B(u)\to +\infty$ as $\|u\|\to +\infty$. Assume that there are two points $v_1,v_2\in X$ such that
\begin{equation}\label{cla}
 c_{\lambda}=\inf_{\gamma\in \Gamma}\sup_{\theta\in [0,1]}\Phi_{\lambda}(\gamma(\theta))>\max\{\Phi_\la(v_1),\Phi_\la(v_1)\},
\ \  \forall \lambda\in T,
\end{equation}
where
\begin{equation}\label{Gamma}
  \Gamma=\{\gamma\in C([0,1],X):\gamma(0)=v_1, \gamma(1)=v_2\}.
\end{equation}
 Then, for almost $\lambda\in T$, there is a sequence $\{u_n(\la)\}\subset X$ such that
\begin{enumerate}[$(a)$]
  \item $\{u_n(\la)\}$ is bounded in $X$;
\end{enumerate}

\begin{enumerate}[$(b)$]
  \item $\Phi_\la(u_n(\la))\to c_\la$ and $\Phi^{\prime}_\la(u_n(\la))\to 0$;
\end{enumerate}

\begin{enumerate}[$(c)$]
  \item the map $\la\to c_\la$ is non-increasing and left continuous.
\end{enumerate}
\end{proposition}

Letting $T=[\delta,1]$, where $\delta\in (0,1)$ is a positive constant. To apply Proposition \ref{proposition}, we will introduce a family of $C^1$ functionals on $X=H^1(\R^3)$ with the form
\begin{equation}\label{IVK}
  I_{V,\lambda}(u)= \frac{1}{2}\int_{\R^3}|\nabla u|^2+V(x)|u|^2dx-\frac{\la}{10}\int_{\R^3} \phi_u |u|^5 dx
 -\la\int_{\R^3}F(u)dx.
\end{equation}
 Then let $I_{V,\lambda}(u)=A(u)-\lambda B(u)$, where
$$
A(u)=\frac{1}{2}\int_{\R^3}|\nabla u|^2+V(x)|u|^2dx\to+\infty \ \  \text{as}\ \ \|u\|\to+\infty,
$$
and
$$
B(u)=\frac{1}{10}\int_{\R^3} \phi_u |u|^5 dx
 +\int_{\R^3}F(u)dx\geq 0.
$$
It is clear that $I_{V, \lambda}$ is a well-defined $C^1$ functional on the space $H^1(\R^3)$, and for all $u,v\in H^1(\R^3)$, one has
\[
\langle I^{\prime}_{V, \lambda}(u),v\rangle= \int_{\R^3}\nabla u\nabla v+V(x)uvdx-\la\int_{\R^3} \phi_u|u|^3 uv dx
-\lambda\int_{\R^3}f(u)vdx.
\]
 \begin{lemma}\label{2Mountpass}
Under the assumptions of $(V_1)-(V_2)$, the function $I_{V, \lambda}$ possesses a Mountain-pass geometry, that is,
\begin{enumerate}[$(a)$]
  \item there exists a $v\in H^1(\R^3)\setminus\{0\}$ such that $I_{V, \lambda}(v)\leq 0$ for all $\lambda\in [\delta,1]$;
  \item $c_{\lambda}=\inf_{\gamma\in\Gamma}\sup_{\theta\in [0,1]}I_{V, \lambda}(\gamma(\theta))>\max\{I_{V, \lambda}(0),I_{V, \lambda}(v)\}$ for all $\lambda\in [\delta,1]$, where
$$
\Gamma=\{\gamma\in C([0,1],H^1(\R^3)):\gamma(0)=0, \gamma(1)=v\}.
$$
\end{enumerate}
\end{lemma}

\begin{proof}
$(a)$ For fixed $u\in H^1(\R^3)\setminus\{0\}$, and any $\lambda\in [\delta,1]$, by $(V_2)$
\[
I_{V, \lambda}(tu)\leq I_{V_\infty, \la}(u)\leq\frac{t^2}{2}\int_{\R^3}|\nabla u|^2dx+\frac{t^2}{2}\int_{\R^3}V_\infty|u|^2dx-\frac{t^{10}}{10}\int_{\R^3}\phi_u |u|^5 dx.
\]
Since $I_{V_\infty,\lambda}(tu)\to-\infty$ as $t\to+\infty$ and we can chose $v=t_0u\in H^1(\R^3)\setminus\{0\}$, then $I_{V,\lambda}(v)<0$ for $t_0$ large enough.

$(b)$ By means of \eqref{Sobolev}, \eqref{growth}, \eqref{Sobolev2} and the Sobolev theorem, one has
\[
I_{V, \lambda}(u)\geq \frac{1}{2}\|u\|^2-\frac{1 }{10S^6}\|u\|^{10}-\epsilon\|u\|^{2}-C_\epsilon\|u\|^{6}.
\]
Let $\epsilon=\frac{1}{4}$, then $I_{V, \lambda}(u)>0$ when $\|u\|=\rho>0$ small.
\end{proof}

Lemma \ref{2Mountpass} and the definition of $I_{V, \lambda}(u)$ imply that $I_{V, \lambda}(u)$ satisfies the assumptions of Proposition \ref{proposition} with $X=H^1(\R^3)$ and $\Phi_\lambda=I_{V, \lambda}$. So for $a.e.$ $\lambda\in [\delta,1]$, there exists a bounded sequence $\{u_n\}\subset H^1(\R^3)$ (here we denote $\{u_n(\lambda)\}$ by $\{u_n\}$ for simplicity)
such that
$$
I_{V, \lambda}(u_n)\to c_\lambda, \ \ I^{\prime}_{V, \lambda}(u_n )\to 0.
$$

By Theorem \ref{maintheorem2}, we infer that for any $\lambda\in [\delta,1]$, the associated limit problem
\begin{equation}\label{mainequation4}
  \left\{%
\begin{array}{ll}
    -\Delta u+V_\infty u- \phi |u|^3u=\lambda f(u), & x\in\R^3, \\
    -\Delta  \phi= |u|^5, &  x\in\R^3.\\
\end{array}%
\right.
\end{equation}
has a least energy solution in $H^1(\R^3)$, $i.e.$ for any $\lambda\in [\delta,1]$,
$$
m^\infty_\lambda\triangleq\inf \big\{   I^{\infty}_{\lambda}(w),\ \ w\in H^1(\R^3)\backslash\{0\} \ \  \text{and}
\ \   P^{\infty}_{\lambda}(w)=0   \big\}
$$
is attained by some $u^\infty_\lambda$, where
\[
I^{\infty}_{\lambda}(u)= \frac{1}{2}\int_{\R^3}|\nabla u|^2+V_\infty|u|^2dx-\frac{1}{10}\int_{\R^3}\phi_u |u|^5dx -\lambda\int_{\R^3}F(u)dx
\]
and
\[
 P^{\infty}_{\lambda}(u) \triangleq\frac{1}{2}\int_{\R^3}|\nabla u|^2dx+\frac{3V_\infty}{2}\int_{\R^3}|u|^2dx-\frac{1}{2}\int_{\R^3} \phi_u|u|^5dx
  -3\lambda\int_{\R^3}F(u)dx\equiv 0.
\]

\begin{lemma}\label{compares}
$c_\la<m_\la^\infty$, where $c_\la$ is given as Proposition \ref{proposition}.
\end{lemma}

\begin{proof}
For any $w\in H^1(\R^3)\backslash\{0\}$ with $P^{\infty}_{\lambda}(w) =0$, then letting $w_t\triangleq w(\frac{x}{t})$ if $t>0$, and $w_t\triangleq w(x)$ if $t=0$. Hence we have that
\begin{eqnarray*}
I^{\infty}_{\lambda}(w_t)&=& I^{\infty}_{\lambda}(w_t)-\frac{t^3}{3}P^{\infty}_{\lambda}(w) \\
   &=&\bigg(\frac{t}{2}-\frac{t^3}{6}\bigg)\int_{\R^3}|\nabla w|^2dx+\bigg(\frac{t^3}{6}-\frac{t^5}{10}\bigg)\int_{\R^3} \phi_{w} |w|^5dx.
\end{eqnarray*}
Clearly, there exists a sufficiently large $T_0>0$ such that $I^{\infty}_{\lambda}(w_{T_0}) <0$ and $I^{\infty}_{\lambda}(w_t)$ achieves the strict global maximum at $t=1$, that is, $I^{\infty}_{\lambda}(w)=\max_{t>0}I^{\infty}_{\lambda}(w_t)$.

On the other hand, we set the path
$$
\gamma_0(t)=
\left\{%
\begin{array}{ll}
   w(\frac{x}{tT_0}) , & t>0, \\
    0, & t=0,\\
\end{array}%
\right.
$$
thus $\gamma_0(0)=0$, $I^{\infty}_{\lambda}(\gamma_0(1))=I^{\infty}_{\lambda}(w_{T_0}) <0$ and then $\gamma_0(t)\in \Gamma$, where $\Gamma$ is given by \eqref{Gamma}. As a consequence of this fact, we can infer that
\begin{eqnarray*}
I^{\infty}_{\lambda}(w)&=& \max_{t\geq0}I^{\infty}_{\lambda}(w_t)
=\max_{t\geq0}I^{\infty}_{\lambda}(w_{tT_0})=\max_{t\geq0}I^{\infty}_{\lambda}\big(\gamma_0(t)\big) \\
   &\geq& \max_{t\in[0,1]}I^{\infty}_{\lambda}\big(\gamma_0(t)\big)
\stackrel{\mathrm{(V_2)}}{>} \max_{t\in[0,1]}I_{V,K,\lambda}\big(\gamma_0(t)\big)  \stackrel{\mathrm{\eqref{Mountainpass1}}}{\geq} c_\la.
\end{eqnarray*}
 Since $w\in H^1(\R^3)\backslash\{0\}$ with $P^{\infty}_{\lambda}(w) =0$ is arbitrary, we have that $c_\la<m_\la^\infty$.
\end{proof}

\begin{lemma}\label{estimateVK}
For $q\in(3,5)$ with any $\mu>0$, or $q\in(1,3]$ with sufficiently large $\mu>0$, the Mountain-pass value given by \eqref{cla} satisfies
\[
   c_{\la}<\frac{2}{5}\la^{-\frac{1}{4}} S^{\frac{3}{2}}
\]
for any $\lambda\in[\delta,1]$ and $S$ is the best Sobolev constant given in \eqref{Sobolev}.
\end{lemma}

\begin{proof}
Firstly, we claim that there exist $t_1,t_2\in(0,+\infty)$ independent of $\epsilon,\lambda$ such that $\max_{t\geq 0}I_{V, \la}(tv_{\epsilon })=I_{V, \la}(t_\epsilon v_{\epsilon })$ and
\begin{equation}\label{2.1g}
  0<t_1<t_\epsilon<t_2<+\infty.
\end{equation}
Indeed, by the facts that $\lim_{t\to +\infty}I_{V, \la}(tv_{\epsilon })=-\infty$ and $(a)$ of Lemma \ref{2Mountpass}, there exists $t_\epsilon>0$ such that
$$
\max\limits_{t\geq 0}I_{V,\la}(tv_{\epsilon })=I_{V, \la}(t_\epsilon v_{\epsilon })>0\ \  \text{and}\ \ \frac{d}{dt}I_{V, \la}(tv_{\epsilon })=0
$$
which imply that
\begin{eqnarray}\label{2.1i}
 \nonumber
  I_{V, \la}(t_\epsilon v_{\epsilon }) &=& \frac{t_\epsilon^2}{2}\|v_{\epsilon }\|^2-\frac{\lambda t_\epsilon^{10}}{10} \int_{\R^3} \phi_{v_{\epsilon }}|v_{\epsilon }|^5dx-\la\int_{\R^3}F(t_\epsilon v_{\epsilon })dx \\
   &\stackrel{\mathrm{(f_1)}}{\leq}& \frac{t_\epsilon^2}{4}\|v_{\epsilon }\|^2-\frac{ t_\epsilon^{10}}{10} S^{-6}\|v_{\epsilon }\|^{10}
    \end{eqnarray}
and
\begin{eqnarray}\label{2.1h}
 \nonumber  t_\epsilon^2\|v_{\epsilon}\|^2&=&\lambda \int_{\R^3}f(t_\epsilon v_{\epsilon})t_\epsilon v_{\epsilon}dx+\la t_\epsilon^{10}\int_{\R^3}\phi_{v_{\epsilon}}|v_{\epsilon}|^5dx \\
   &\geq&  \delta\mu t_\epsilon^{q+1}|v_{\epsilon}|_{q+1}^{q+1}+\delta t_\epsilon^{10}\int_{\R^3}\phi_{v_{\epsilon}}|v_{\epsilon}|^5dx.
\end{eqnarray}
 The fact $I_{V,\la}(t_\epsilon v_{\epsilon})>0$ yields that $t_\epsilon$ is bounded from below by \eqref{2.1i}. And since $q+1>2$, then it follows from \eqref{2.1h} that $t_\epsilon$ is bounded from above.
  Hence \eqref{2.1g} is true.

Let us define
 \begin{eqnarray*}
g(t) &:=& \frac{t^2}{2}\int_{\R^3}|\nabla v_{\epsilon }|^2dx-\frac{ \la t^{10}}{10}\int_{\R^3} \phi_{v_{\epsilon}}|v_{\epsilon }|^5dx \\
   &:=& C_1t^2-C_2t^{10},
\end{eqnarray*}
where
$$
C_1=\frac{1}{2}\int_{\R^3}|\nabla v_{\epsilon}|^2dx,\ \ C_2=\frac{\la}{10}\int_{\R^3} \phi_{v_{\epsilon}}|v_{\epsilon }|^5dx.
$$
By some elementary calculations, we have
\begin{equation}\label{ggggg}
  \max_{t\geq 0}g(t) = \frac{4(C_1)^{\frac{5}{4}}}{5(5C_2)^{{\frac{1}{4}}}}
   = \frac{2}{5}\la^{-{\frac{1}{4}}}\frac{   \Big( \int_{\R^3}|\nabla v_{\epsilon}|^2dx \Big)^{  \frac{5}{4}    }    }      { \Big(  \int_{\R^3}\phi_{v_{\epsilon}}|v_{\epsilon}|^5dx  \Big)^{  \frac{1}{4}    }       }.
\end{equation}
The Poisson equation $-\Delta \phi_{v_\epsilon}=|v_{\epsilon}|^5$ and Cauchy's inequality give
\begin{align*}
 \int_{\R^3} |v_{\epsilon}|^6dx &=  \int_{\R^3}\nabla \phi_{v_\epsilon} \nabla |v_{\epsilon}| dx
   \leq \frac{1}{2} \int_{\R^3}|\nabla \phi_{v_\epsilon}|^2dx+ \frac{1}{2} \int_{\R^3}|\nabla {v_{\epsilon}}|^2dx\\
   &=\frac{1}{2} \int_{\R^3}\phi_{v_{\epsilon}}|{v_{\epsilon}}|^5dx+ \frac{1}{2} \int_{\R^3}|\nabla {v_{\epsilon}}|^2dx
\end{align*}
which implies that
\begin{eqnarray*}
\int_{\R^3}\phi_{v_{\epsilon}}|{v_{\epsilon}}|^5dx &\geq& 2\int_{\R^3} |v_{\epsilon}|^6dx- \int_{\R^3}|\nabla {v_{\epsilon}}|^2dx \\
&\stackrel{\mathrm{\eqref{estimate1}}}{=}&S^{\frac{3}{2}}+O(\epsilon).\ \ \ \
\end{eqnarray*}
As a consequence of the above fact, one has
\begin{equation}\label{gggggggggggg}
 \max_{t\geq 0}g(t)\stackrel{\mathrm{\eqref{ggggg}}}{\leq} \frac{2}{5}\la^{-\frac{1}{4}}\frac   { \Big(S^{\frac{3}{2}}+O(\epsilon)\Big)^{   \frac{5}{4}  }   }
{  \Big(S^{\frac{3}{2}}+O(\epsilon)\Big)^{   \frac{1}{4}  }   }
=\frac{2}{5}\la^{-\frac{1}{4}}S^{\frac{3}{2}}+O(\epsilon).
\end{equation}
On the other hand, for $\epsilon>0$ with $\epsilon<1$ we have
 \begin{eqnarray}\label{ggggggggggggg}
\nonumber  \la \int_{\R^3}F(tv_{\epsilon})dx &\stackrel{\mathrm{(f_3)}}{\geq }& \delta\mu t_\epsilon^{q+1}\int_{\R^3}|v_{\epsilon}|^{q+1}dx
    \stackrel{\mathrm{\eqref{2.1g}}}{\geq } C\mu \int_{\R^3}|v_{\epsilon}|^{q+1}dx\\
    &\stackrel{\mathrm{\eqref{estimate2}}}{=}&\left\{
  \begin{array}{ll}
   C\mu O(\epsilon^{\frac{q+1}{2}}), & \text{if}~q\in[1,2), \\
    C\mu O(\epsilon^{\frac{3}{2}}|\log \epsilon|), & \text{if}~q=2, \\
    C\mu O(\epsilon^{\frac{5-q}{2}}), &   \text{if}~q\in(2,5).
  \end{array}
\right.
\end{eqnarray}
 We have proved $\max_{t\geq 0}I(tv_\epsilon)=I(t_\epsilon v_\epsilon)$ at the beginning, that is,
\begin{eqnarray}\label{zzzzz}
\nonumber \max_{t\geq 0}I(tv_{\epsilon})&=& \frac{t_\epsilon^2}{2}\int_{\R^3}|\nabla v_{\epsilon}|^2+V(x)|v_{\epsilon}|^2dx
-\frac{\la t_\epsilon^{10}}{10}\int_{\R^3}\phi_{v_{\epsilon}}|v_{\epsilon}|^5dx\\
\nonumber &&-\la \int_{\R^3}F(tv_{\epsilon})dx  \\
\nonumber    &\stackrel{\mathrm{(V_2)}}{\leq}&g(t_\epsilon)+\frac{V_\infty t_\epsilon^2}{2}|v_{\epsilon}|_2^2- \la\int_{\R^3}F(tv_{\epsilon})dx \\
   &\leq&  \frac{2}{5}\la^{-\frac{1}{4}}S^{\frac{3}{2}}+CO(\epsilon)-\la \int_{\R^3}F(tv_{\epsilon})dx,
\end{eqnarray}
where we have used \eqref{estimate2}, \eqref{2.1g} and \eqref{gggggggggggg} in the last inequality.

Using \eqref{ggggggggggggg}, one has
\[
CO(\epsilon)-\la \int_{\R^3}F(tv_{\epsilon})dx\leq CO(\epsilon)-
\left\{
  \begin{array}{ll}
 C\mu  O(\epsilon^{\frac{q+1}{2}}), & \text{if}~q\in[1,2), \\
   C\mu  O(\epsilon^{\frac{3}{2}}|\log \epsilon|), & \text{if}~q=2, \\
   C\mu  O(\epsilon^{\frac{5-q}{2}}), &   \text{if}~q\in(2,5).
  \end{array}
\right.
 \]

If $q\in(3,5)$, then $0<\frac{5-q}{2}<1$ and for sufficiently small $\epsilon>0$
\[
CO(\epsilon)- \mu\int_{\R^3}F(tv_{\epsilon})dx\leq CO(\epsilon)-C\mu O(\epsilon^{\frac{5-q}{2}})<0
\]
for any $\mu>0$.

If $q\in (2,3]$ and $\mu=\epsilon^{-\frac{1}{2}}$, then $\frac{1}{2}\leq\frac{5-q}{2}-\frac{1}{2}<1$ and hence for sufficiently small $\epsilon>0$
\[
CO(\epsilon)- \mu\int_{\R^3}F(tv_{\epsilon})dx\leq CO(\epsilon)-C\epsilon^{-\frac{1}{2}} O(\epsilon^{\frac{5-q}{2}})<0
\]
when $\mu>0$ is sufficiently large.

If $q=2$ and $\mu=\epsilon^{-\frac{1}{2}}$, then for sufficiently small $\epsilon>0$
\[
CO(\epsilon)-\mu\int_{\R^3}F(tv_{\epsilon})dx\leq CO(\epsilon)-C\epsilon^{-\frac{1}{2}} O(\epsilon^{\frac{3}{2}}|\log \epsilon|)<0
\]
when $\mu>0$ is sufficiently large.

If $q\in (1,2)$ and $\mu=\epsilon^{-\frac{1}{2}}$, then $\frac{1}{2}<\frac{q+1}{2}-\frac{1}{2}<1$ and hence for sufficiently small $\epsilon>0$
\[
CO(\epsilon)-\mu\int_{\R^3}F(tv_{\epsilon})dx\leq CO(\epsilon)-C\epsilon^{-\frac{1}{2}} O(\epsilon^{\frac{q+1}{2}})<0
\]
when $\mu>0$ is sufficiently large.

Consequently, we have showed that for $q\in(3,5)$ with any $\mu>0$, or $q\in(1,3]$ with sufficiently large $\mu>0$
\[
CO(\epsilon)- \int_{\R^3}F(tv_{\epsilon})dx<0
 \]
which indicates that $c_\la<\frac{2}{5}\la^{-\frac{1}{4}}S^{\frac{3}{2}}$ by \eqref{compare} and \eqref{zzzzz}.
\end{proof}

A bounded $(PS)_{c_\lambda}$ is constructed and to get a critical point of the functional $I_{V, \lambda}$, the
following version of a global compactness lemma is necessary to prove that the functional $I_{V, \lambda}$ satisfies $(PS)_{c_\lambda}$ condition for $a.e.$ $\lambda\in [\delta,1]$.

\begin{lemma}\label{global}
Assume that $(V_1)-(V_2)$ and $(f_1)-(f_3)$ hold, for any $\lambda \in [\delta,1]$ and let $\{u_n\}$ be a bounded $(PS)_{c_\lambda}$ sequence for the functional $I_{V, \lambda}$ with
$$
c_\la<\frac{2}{5}\la^{-\frac{1}{4}} S^{\frac{3}{2}}.
$$
Then there exist a subsequence of $\{u_n\}$ still denoted by itself, an integer $k\in \mathbb{N}\cup\{0\}$, a sequence $\{y^i_n\}\subset\R^3$, and $w_i\in H^1(\R^3)$ for $i\in\{1,2,\cdots,k\}$ such that
\begin{enumerate}[$(a)$]
  \item  $u_n\rightharpoonup u_0$ with $I^{\prime}_{V, \lambda}(u_0)=0$;
  \item  $|y^i_n|\to +\infty$ and $|y^i_n-y^j_n|\to+\infty$ if $i\neq j$;
  \item  $w_i\neq 0$ and $(I^{\infty}_{\lambda})^{\prime}(w_i)=0$ for $i\in\{1,2,\cdots,k\}$;
  \item  $\|u_n-u_0-\sum_{i=1}^{i=k}w_i(\cdot-y^i_n)\|\to 0$;
  \item  $I_{V, \lambda}(u_n)\to I_{V, \lambda}(u_0)+ \sum_{i=1}^{i=k}I_{\lambda}^\infty(w_i)$.
\end{enumerate}
\end{lemma}

\begin{proof}
Since $\{u_n\}$ is bounded in $H^1(\R^3)$, going to a subsequence if necessary $u_n\rightharpoonup u_0$ and we have $I^{\prime}_{V, \lambda}(u_0)=0$ as the Step 2 in the proof of Theorem \ref{maintheorem2}. On the other hand we claim that $I_{V, \lambda}(u_0)\geq 0$. In fact, since $I^{\prime}_{V, \lambda}(u_0)=0$, then $P_{V, \la}(u_0)=0$, where $P_{V, \la}(u)$ is given as similar to \eqref{Pohozaev1}. To simply the calculations, let us
introduce the following notations:
\begin{equation}\label{simple121}
 \alpha\triangleq\int_{\R^3}|\nabla u_0|^2dx,\,\ \beta\triangleq\int_{\R^3}V(x)|u_0|^2dx,\ \ \overline{\beta}\triangleq\int_{\R^3}(x,\nabla V)|u_0|^2dx,
\end{equation}
\begin{equation}\label{simple122}
  \delta\triangleq\la\int_{\R^3} \phi_u|u|^5dx,
  \ \ \kappa\triangleq \la\int_{\R^3}F(u)dx,
\end{equation}
hence we have
\begin{equation}\label{simple123}
  \left\{
  \begin{array}{ll}
  I_{V, \lambda}(u_0)=\frac{1}{2} \alpha+\frac{1}{2}\beta-\frac{1}{10}\delta-\kappa,\\
P_{V, \lambda}(u_0)=\frac{1}{2} \alpha+\frac{3}{2}\beta+\frac{1}{2}\overline{\beta}-\frac{1}{2}\delta
-3\kappa=0.
  \end{array}
\right.
\end{equation}
By means of the H\"{o}lder inequality, \eqref{Sobolev} and the assumption $(V_1)$, one has
\begin{eqnarray}\label{simple124}
\nonumber|\overline{\beta}|&=&\bigg|\int_{\R^3}(x,\nabla V)|u_0|^2dx\bigg| \leq \bigg(\int_{\R^3}|(x,\nabla V)|^{\frac{3}{2}}\bigg)^{\frac{2}{3}}  \bigg(\int_{\R^3}|u_0|^{6}\bigg)^{\frac{1}{3}}\\
  &\leq& S S^{-1}\int_{\R^3}|\nabla u_0|^2dx= \alpha.
\end{eqnarray}
It follows from \eqref{simple121}-\eqref{simple124} that
\begin{eqnarray}\label{gggggg}
\nonumber  I_{V, \lambda}(u_0)&=&I_{V, \lambda}(u_0)- \frac{1}{3}P_{V, \lambda}(u_0)\\
 &=&\frac{1}{3}\alpha-\frac{1}{6}\overline{\beta}+\frac{1}{15}\delta
 \geq  \frac{1}{6} \alpha,
\end{eqnarray}
which yields that $I_{V,\lambda}(u_0)\geq 0$.

Let us set $u_n^1\triangleq u_n-u_0$, then $u_n^1\rightharpoonup 0$ in $H^1(\R^3)$. Hence it follows from \eqref{weak3} and Br\'{e}zis-Lieb lemma \cite{Brezis2} that
 \begin{equation}\label{strong1}
\left\{%
\begin{array}{ll}
  \displaystyle   \|u_n\|^2-\|u_n^1\|^2-\|u_0\|^2\to 0, \\
 \displaystyle    \int_{\R^3} \phi_{u_n}|u_n|^5dx-\int_{\R^3} \phi_{u_n^1}|u_n^1|^5dx-\int_{\R^3} \phi_{u_0}|u_0|^5dx\to 0,\\
 \displaystyle   \ \   \int_{\R^3}F(u_n)dx-\int_{\R^3}F(u_n^1)dx-\int_{\R^3}F(u_0)dx\to0,\\
\end{array}%
\right.
\end{equation}
which reveals that
\begin{equation}\label{strong2}
  I_{V, \la}(u_n)-I_{V, \la}(u_n^1)-I_{V, \la}(u_0)\to 0
\end{equation}
and
\begin{equation}\label{strong3}
  \langle I_{V, \la}^{\prime}(u_n),u_n\rangle-\langle I_{V, \la}^{\prime}(u_n^1),u_n^1\rangle-\langle I_{V, \la}^{\prime}(u_0),u_0\rangle\to 0.
\end{equation}
 Define
\[
\sigma\triangleq \mathop{\lim\sup}_{n\to\infty}\sup_{y\in\R^3}\int_{B_1(y)}|u_n^1|^2dx\geq 0.
\]
We will consider the cases $\sigma=0$ and $\sigma>0$, respectively.
\vskip0.3cm
\underline{\textbf{Case 1.}}\ \   $\sigma=0$.
\vskip0.3cm
\noindent Using Lemma \ref{Vanishing}, we have that $u_n^1\to 0$ in $L^p(\R^3)$ for any $2<p<6$. Thus using Strass's compactness lemma in \cite{Berestycki} and the assumptions $(f_1)-(f_2)$ again, one has
\[
\int_{\R^3}F(u_n^1)dx\to 0\ \ \text{and}\ \ \int_{\R^3}u_n^1f(u_n^1)dx\to 0.
\]
By means of \eqref{strong2}-\eqref{strong3} and the above formulas, we have the following results at once:
\[
c_\la-I_{V, \la}(u_0)+o(1)=\frac{1}{2}\|u_n^1\|^2-\frac{\la}{10}\int_{\R^3} \phi_{u_n^1}|u_n^1|^5dx
\]
and
\[
o(1)=\|u_n^1\|^2-\la\int_{\R^3} \phi_{u_n^1}|u_n^1|^5dx.
\]
Hence without loss of generality, we may assume
\[
\lim_{n\to\infty}\|u_n^1\|^2=\lim_{n\to\infty}\la\int_{\R^3} \phi_{u_n^1}|u_n|^5dx=l^1,\ \  \text{and}\ \ c_\la-I_{V ,\la}(u_0)=\frac{2}{5}l^1.
\]
On the other hand, by \eqref{Sobolev2} we can deduce that
\[
\la\int_{\R^3} \phi_{u_n^1}|u_n^1|^5dx\leq \la  S^{-6}\|u_n^1\|^{10}
\]
which implies that $l^1\leq \la S^{-6}(l^1)^5$. Hence either $l^1=0$ or $l^1\geq \la^{-\frac{1}{4}} S^{\frac{3}{2}}$. If we suppose $l^1\geq \la^{-\frac{1}{4}} S^{\frac{3}{2}}$, then in view of the fact that $I_{V,\la}(u_0)\geq 0$, we can infer
\[
c_\la\geq c_\la-I_{V,\la}(u_0) \geq \frac{2}{5} \la^{-\frac{1}{4}} S^{\frac{3}{2}}
\]
which yields a contradiction to Lemma \ref{estimateVK}. Thus $l^1=0$, $i.e.$ $\lim_{n\to\infty}\|u_n^1\|\to0$.
\vskip0.3cm
\underline{\textbf{Case 2.}}\ \   $\sigma>0$.
\vskip0.3cm
\noindent We may assume that there exists $y_n^1\in\R^3$ such that
\begin{equation}\label{strong4}
\int_{B_1(y_n^1)}|u_n^1|^2dx>\frac{\sigma}{2}> 0.
\end{equation}
Let's define $\widetilde{u}_n^1(x)=u_n^1(x+y_n^1)$ and clearly $\{\widetilde{u}_n^1\}$ is bounded in $H^1(\R^3)$ and we may suppose that $\widetilde{u}_n^1\rightharpoonup w_1$ in $H^1(\R^3)$, $\widetilde{u}_n^1\to w_1$ in $L_{\text{loc}}^p(\R^3)$ and $\widetilde{u}_n^1\to w_1$ $a.e.$ in $\R^3$. As a consequence of \eqref{strong4}, we get
\[
\int_{B_1(0)}|w_1|^2dx>\frac{\sigma}{2}> 0
\]
which implies that $w_1\neq 0$. Since $u_n^1\rightharpoonup 0$ in $H^1(\R^3)$, then $\{y_n^1\}$ must be unbounded. Going to a subsequence if necessary, $|y_n^1|\to+\infty$.

We now show that $(I^{\infty}_{\lambda})^{\prime}(w_1)=0$.  In fact, it is easy to show that
\begin{equation}\label{Feb1}
 \big\langle(I^{\infty}_{\lambda})^{\prime}(\widetilde{u}_n^1),\varphi\big\rangle-  \big\langle(I^{\infty}_{\lambda})^{\prime}(w_1),\varphi\big\rangle\to 0
\end{equation}
for any $\varphi\in C_0^\infty(\R^3)$. Since $u_n^1\rightharpoonup 0$ in $H^1(\R^3)$, we obtain
\[
\big\langle I^{\prime}_{V,\lambda}({u}_n^1),\varphi(\cdot-y_n^1)\big\rangle-  \big\langle I_{V,\lambda}^{\prime}(0),\varphi(\cdot-y_n^1)\big\rangle\to 0
\]
which yields that
\begin{equation}\label{Feb2}
  \big\langle I^{\prime}_{V,\lambda}({u}_n^1),\varphi(\cdot-y_n^1)\big\rangle\to 0.
\end{equation}
 Using $(V_2)$ and $|y_n^1|\to+\infty$, one has
 \begin{equation}\label{Feb3}
    \int_{\R^3}V(x+y_n^1)\widetilde{u}_n^1(x)\varphi(x)dx-\int_{\R^3}V_\infty\widetilde{u}_n^1(x)\varphi(x)dx\to 0.
 \end{equation}
Combing \eqref{Feb2} and \eqref{Feb3}, one has
\[
\big\langle(I^{\infty}_{\lambda})^{\prime}(\widetilde{u}_n^1),\varphi\big\rangle\to 0
\]
which implies that $\big\langle(I^{\infty}_{\lambda})^{\prime}(w_1),\varphi\big\rangle=0$ for any $\varphi\in C_0^\infty(\R^3)$ together with \eqref{Feb1}.

In view of $(V_2)$ and the locally compactness of Sobolev embedding, one has
\begin{equation}\label{Feb33}
  \int_{\R^3}\big[V_\infty-V(x)\big]|u_n^1|^2dx=
\int_{\R^3}\big[V_\infty-V(x)\big]|u_n-u_0|^2dx
\to 0.
\end{equation}
Hence it follows from \eqref{strong1} and \eqref{strong2} that
\begin{equation}\label{Feb4}
  I_{V,\la}(u_n)-I_{V,\la}(u_0)-I_{\la}^\infty(u_n^1)\to 0.
\end{equation}

Let us set $u^2_n(\cdot)\triangleq u^1_n(\cdot)-w_1(\cdot-y^1_n)$, then $u^2_n\rightharpoonup 0$ in $H^1(\R^3)$. Thus it follows from \eqref{weak3} and Br\'{e}zis-Lieb lemma \cite{Brezis2} that
 \begin{equation}\label{Feb5}
\left\{%
\begin{array}{ll}
  \displaystyle   \|u_n^2\|^2-\|u_n\|^2-\|u_0\|^2-\|w_1(\cdot-y_n^1)\|\to 0, \\
 \displaystyle    \int_{\R^3} \phi_{u_n^2}|u_n^2|^5dx-\int_{\R^3} \phi_{u_n}|u_n|^5dx-\int_{\R^3} \phi_{u_0}|u_0|^5dx-\int_{\R^3} \phi_{w_1}|w_1|^5dx\to 0,\\
 \displaystyle   \ \   \int_{\R^3}F(u_n^2)dx-\int_{\R^3}F(u_n)dx-\int_{\R^3}F(u_0)dx-\int_{\R^3}F(w_1(\cdot-y_n^1))dx\to0,\\
\end{array}%
\right.
\end{equation}
\begin{equation}\label{Feb6}
  \begin{gathered}
  \int_{\R^3} \phi_{u_n^2}|u_n^2|^3u_n^2\varphi dx-\int_{\R^3} \phi_{u_n}|u_n|^3u_n\varphi dx-\int_{\R^3} \phi_{u_0}|u_0|^3u_0\varphi dx\hfill\\
  -\int_{\R^3} \phi_{w_1(x-y_n^1)}|w_1(x-y_n^1)|^3w_1(x-y_n^1)\varphi dx \to 0
  \end{gathered}
\end{equation}
and
\begin{equation}\label{Feb7}
  \begin{gathered}
  \int_{\R^3} V(x)|u_n^2|^2 dx-\int_{\R^3} V(x)|u_n|^2 dx-\int_{\R^3}V(x)|u_0|^2 dx\hfill\\
  -\int_{\R^3} V(x)|w_1(x-y_n^1)|^2 dx \to 0.
  \end{gathered}
\end{equation}
By means of \eqref{Feb5}, \eqref{Feb6} and \eqref{Feb7}, we derive
 \begin{equation}\label{Feb8}
\left\{%
\begin{array}{ll}
  \displaystyle   I_{V,\la}(u_n^2)-I_{V,\la}(u_n)-I_{V,\la}(u_0)-I_{V,\la}(w_1)\to 0, \\
 \displaystyle    I_\la^\infty(u_n^2)-I_{V,\la}(u_n^1)-I_\la^\infty(w_1)\to 0,\\
 \displaystyle   \ \   \big\langle I^{\prime}_{V,\lambda}({u}_n^2),\varphi\big\rangle-\big\langle I^{\prime}_{V,\lambda}({u}_n),\varphi\big\rangle-\big\langle I^{\prime}_{V,\lambda}({u}_0),\varphi\big\rangle-\big\langle (I_{\lambda}^\infty)^{\prime}({w}_1),\varphi\big\rangle
 \to 0.\\
\end{array}%
\right.
\end{equation}
Thus by using \eqref{Feb33}, \eqref{Feb4} and \eqref{Feb8}, we have
\begin{align*}
   I_{V,\la}(u_n)&= I_{V,\la}(u_0)+I_\la^\infty(u_n^1) +o(1)\\
   & =I_{V,\la}(u_0)+I_\la^\infty(w_1)+I_\la^\infty(u^2_n)+o(1).
\end{align*}
Recalling that $I_\la^\infty(w_1)\geq 0$ and $I_{V,\la}(u_0)\geq 0$, then we can conclude that
\[
I_{V,\la}(u_n^2)=I_{V,\la}(u_n)-I_{V,\la}(u_0)-I_\la^\infty(w_1)+o(1)\leq c_\la<\frac{2}{5}\la^{-\frac{1}{4}}S^{\frac{3}{2}}.
\]
Using the same arguments as before, let
\[
\sigma_1\triangleq\mathop{\lim\sup}_{n\to\infty}\sup_{y\in\R^3}\int_{B_1(y)}|u_n^2|^2dx\geq 0.
\]

If $\sigma_1=0$, then $\|u_n^2\|\to 0$, that is, $\|u_n-u_0-w_1(\cdot-y_n^1)\|\to 0$ and hence the proof is complete with $k=1$.

If $\sigma_1>0$, then there exists a sequence $\{y_n^2\}\subset\R^3$ and $w_2\in H^1(\R^3)$ such that $\widetilde{u}_n^2(x)\triangleq u_n^2(x+y_n^2)\rightharpoonup w_2$ in $H^1(\R^3)$. By \eqref{Feb8}, one has $(I_\la^\infty)^\prime(w_2)=0$. On the other hand, $u_n^2\rightharpoonup 0$ in $H^1(\R^3)$ yields that $|y_n^2|\to+\infty$ and $|y_n^1-y_n^2|\to+\infty$. Consequently iterating this procedure we can obtain there are sequences $\{y_n^i\}\subset\R^3$ such that $|y_n^i|\to+\infty$ and $|y_n^i-y_n^j|\to+\infty$ if $i\neq j$ and ${u}_n^i(x)\triangleq u_n^{i-1}-w_{i-1}(\cdot-y_n^{i-1})$ with $i\geq 2$ such that
\[
u_n^i\rightharpoonup 0 \ \  \text{in}\ \ H^1(\R^3)\ \  \text{and}\ \
(I^\infty_\la)^{\prime}(w_i)=0
\]
and
 \begin{equation}\label{Feb12}
\left\{%
\begin{array}{ll}
  \displaystyle   \|u_n\|^2-\|u_0\|^2-\sum_{i=1}^{k-1}\|w_i(\cdot-y_n^i)\|^2
  =\bigg\|u_n-u_0-\sum_{i=1}^{k-1}w_i(\cdot-y_n^i)\bigg\|^2+o(1) \\
 \displaystyle   I_{V,\la}(u_n)=I_{V,\la}(u_0)+\sum_{i=1}^{k-1}I_\la^\infty(w_i)+I_\la^\infty(u_n^k)+o(1).  \\
\end{array}%
\right.
\end{equation}
Since $\{u_n\}$ is bounded in $H^1(\R^3)$, then the iteration stops at some finite index $k+1$ by \eqref{Feb12}. Consequently $u_n^{k+1}\to 0$ in $H^1(\R^3)$ and using \eqref{Feb12} again, we have the conclusions $(d)$ and $(e)$. The proof is completed.
\end{proof}

Using Lemma \ref{global}, we can prove the functional $I_{V,\la}$ satisfies the so-called $(PS)_{c_\la}$ condition, that is,
\begin{lemma}\label{PScondition}
Assume that $(V_1)-(V_2)$ and $(f_1)-(f_3)$ hold, for $\lambda\in [\delta, 1]$ and let $\{u_n\}\subset H^1(\R^3)$ be a bounded $(PS)_{c_\lambda}$ sequence of $I_{V,\lambda}$. Then there exists a nontrivial $u_\lambda\in H^1(\R^3)$ such that
\[
u_n\to u_\lambda\ \ \text{in}\ \ H^1(\R^3).
\]
\end{lemma}

\begin{proof}
According to Lemma \ref{global}, we know that for $\lambda\in [\delta, 1]$, there exists a $u_\lambda\in H^1(\R^3)$ such that
\[
u_n\rightharpoonup u_\lambda,\ \ I^{\prime}_{V, \lambda}(u_\lambda)=0.
\]
In order to end the proof, it is enough to show that $k\equiv 0$ by Lemma \ref{global} $(d)$. To show it, we argue by contradiction and suppose that $k>0$. By contradiction, we assume there exist sequences $\{y^i_n\}\subset \R^3$ with $|y^i_n|\to +\infty$ for every $i\in\{1,2,\cdots,k\}$ such that
\[
 I_{V, \lambda}(u_n)\to I_{V, \lambda}(u_\lambda)+ \sum_{i=1}^{i=k}I_{\lambda}^\infty(w_i),
\]
 where $w_i$ is critical point of $I_{\lambda}^\infty$ for $i\in\{1,2,\cdots,k\}$. It is clear that $I_{\lambda}^\infty(w_i)\geq m_\lambda^\infty$ and the last thing is to show $I_{V, \lambda}(u_\lambda)\geq 0$.

We give the following notations for simplicity
\begin{equation}\label{simple21}
 \alpha_\la\triangleq\int_{\R^3}|\nabla u_\la|^2dx,\,\ \beta_\la\triangleq\int_{\R^3}V(x)|u_\la|^2dx,\ \ \overline{\beta}_\la\triangleq\int_{\R^3}(x,\nabla V)|u_\la|^2dx,
\end{equation}
\begin{equation}\label{simple22}
 \delta_\la\triangleq\la\int_{\R^3} \phi_{u_\la}|u_\la|^5dx,
  \ \ \kappa_\la\triangleq \la\int_{\R^3}F(u_\la)dx,
\end{equation}
hence we have
\begin{equation}\label{simple23}
\left\{
  \begin{array}{ll}
  I_{V, \lambda}(u_\la)=\frac{1}{2} \alpha_\la+\frac{1}{2}\beta_\la-\frac{1}{10}\delta_\la-\kappa_\la,\\
P_{V, \lambda}(u_\la)=\frac{1}{2} \alpha_\la+\frac{3}{2}\beta_\la+\frac{1}{2}\overline{\beta}_\la-\frac{1}{2}\delta_\la
-3\kappa_\la=0.
  \end{array}
\right.
\end{equation}
 By means of the H\"{o}lder inequality, \eqref{Sobolev} and the assumption $(V_1)$, one has
\begin{eqnarray}\label{simple24}
\nonumber|\overline{\beta}_\la|&=&\bigg|\int_{\R^3}(x,\nabla V)|u_\la|^2dx\bigg| \leq \bigg(\int_{\R^3}|(x,\nabla V)|^{\frac{3}{2}}\bigg)^{\frac{2}{3}}  \bigg(\int_{\R^3}|u_\la|^{6}\bigg)^{\frac{1}{3}}\\
  &\leq& S S^{-1}\int_{\R^3}|\nabla u_\la|^2dx= \alpha_\la.
\end{eqnarray}
It follows from \eqref{simple21}-\eqref{simple24} that
\begin{eqnarray*}
\nonumber  I_{V, \lambda}(u_\la)&=&I_{V, \lambda}(u_\la)- \frac{1}{3}P_{V, \lambda}(u_\la)\\
 &=&\frac{1}{3}\alpha_\la-\frac{1}{6}\overline{\beta}_\la+\frac{1}{15}\delta_\la
 \geq  \frac{1}{6} \alpha_\la,
\end{eqnarray*}
which gives that $I_{V,\lambda}(u_\lambda)\geq 0$. Therefore, we infer that
\[
c_\lambda=\lim_{n\to\infty}I_{V, \lambda}(u_{n})=I_{V, \lambda}(u_\lambda)+\sum_{i=1}^k I^\infty_\lambda(w_i)\geq m_\lambda^\infty,
\]
which yields a contradiction in view of Lemma \ref{compares}. So $u_n\to u_\lambda$ in $H^1(\R^3)$. Consequently, we obtain $I_{V, \lambda}(u_\lambda)=c_\lambda$ and $I^{\prime}_{V, \lambda}(u_\lambda)=0$.
 \end{proof}

Now we will establish a least energy solution for the system \eqref{mainequation1}, we define
\[
m_{V }\triangleq\inf_{u\in \mathcal{S}_{V }}I_{V }(u),
\]
where $\mathcal{S}_{V }\triangleq\{u\in H^1(\R^3)\backslash\{0\}:I^{\prime}_{V }(u)=0\}$.
\vskip3mm
\begin{lemma}\label{S}
$\mathcal{S}_{V }\neq \emptyset$.
\end{lemma}

\begin{proof}
It follows from Proposition \ref{proposition} and Lemma \ref{2Mountpass} that for almost everywhere $\lambda\in[\delta, 1]$, there exists a bounded sequence $\{u_n\}\subset H^1(\R^3)$ such that
$$
I_{V, \lambda}(u_n )\to c_\lambda, \ \ I^{\prime}_{V, \lambda}(u_n )\to 0.
$$
Recalling Lemma \ref{PScondition}, $I_{V, \lambda}$ has a nontrivial critical point $u_\lambda\in H^1(\R^3)$ and $I_{V, \lambda}(u_\lambda)=c_\lambda$. By means of the above discussions again and again, there exists a sequence $\{\lambda_n\}\subset [\delta, 1]$ with $\lambda_n\to1^-$ and an associated sequence $\{u_{\lambda_n}\}\subset H^1(\R^3)$ such that $I_{V, \lambda_n}(u_{\lambda_n})=c_{\lambda_n}$ and $I^{\prime}_{V, \lambda}(u_{\lambda_n})=0$.
Note that the Lemma \ref{Pohozaev}, then $P_{V, \la_n}(u_{\lambda_n})=0$. Hence it is similar to the proof of Theorem \ref{maintheorem2} that $\{u_{\lambda_n}\}$ is bounded is $H^1(\R^3)$.

Since $\lambda_n\to1^-$, we claim that $\{u_{\lambda_n}\}$ is a $(PS)_{c_1}$ sequence of $I_{V }=I_{V, 1}$. In fact, as a consequence of Proposition \ref{proposition} (c) we obtain that
\[
\lim_{n\to\infty} I_{V, 1}(u_{\lambda_n})=\bigg(\lim_{n\to\infty}I_{V, \lambda_n}(u_{\lambda_n})
+ (\lambda_n-1)\int_{\R^3}F(u_{\lambda_n})dx\bigg)=\lim_{n\to\infty}c_{\lambda_n}=c_1
\]
and for all $\psi\in H^1(\R^3)\backslash\{0\}$,
\begin{eqnarray*}
 \lim_{n\to\infty} \frac{|\langle I^{\prime}_{V, 1}(u_{\lambda_n}),\psi\rangle|}{\|\psi\|}&=& \lim_{n\to\infty}\frac{\big|\langle I^{\prime}_{V, \lambda_n}(u_{\lambda_n}),\psi\rangle
+ (\lambda_n-1) \int_{\R^3}f(u_{\lambda_n})\psi dx\big|}{\|\psi\|} \\
   &\leq&\lim_{n\to\infty} \frac{|\lambda_n-1|\int_{\R^3}|f(u_{\lambda_n})||\psi| dx}{\|\psi\|} \\
   &\stackrel{\mathrm{\eqref{growth}}}{\leq} & \lim_{n\to\infty}  \frac{|\lambda_n-1|}{\|\psi\|}\bigg( \epsilon|u_{\lambda_n}|_{2}|\psi|_2+ C_\epsilon|u_{\lambda_n}|_{6}^5|\psi|_6 \bigg) \to 0,
\end{eqnarray*}
which imply that $\{u_{\lambda_n}\}$ is a $(PS)_{c_1}$ sequence of $I_{V }=I_{V, 1}$, where we have used the fact that $\{u_{\lambda_n}\}$ is bounded in $H^1(\R^3)$. Then by Lemma 3.5, $I_{V }$ has a nontrivial critical point, that is, $I_{V }(u_0)=c_1$ and $I^{\prime}_{V }(u_0)=0$. The proof is complete.
\end{proof}

\begin{lemma}\label{m}
For $q\in(3,5)$ with any $\mu>0$, or $q\in(1,3]$ with sufficiently large $\mu>0$, there holds
$$
0 <m_{V } <\frac{2}{5} S^{\frac{3}{2}}.
$$
\end{lemma}
\begin{proof}
It follows from Lemma \ref{estimateVK} that $m_{V } <\frac{2}{5} S^{\frac{3}{2}}$ when $q\in(3,5)$ with any $\mu>0$, or $q\in(1,3]$ with sufficiently large $\mu>0$. Next we show $m>0$,
for any $u\in \mathcal{S}_{V }$, the Pohoz\v{a}ev identity \eqref{Pohozaev1} holds, that is, $P_{V }(u) =0$. Therefore we have that
\begin{eqnarray}\label{mm}
\nonumber I_{V }(u)&=& I_{V }(u)-\frac{1}{3}P_{V }(u)\\
\nonumber&=&\frac{1}{3}\int_{\R^3}|\nabla u|^2dx-\frac{1}{6}\int_{\R^3}(x,\nabla V)|u|^2dx+ \frac{1}{15}\int_{\R^3} \phi_u|u|^5dx \\
\nonumber&\geq&\frac{1}{3}\int_{\R^3}|\nabla u|^2dx-\frac{1}{6}\int_{\R^3}(x,\nabla V)|u|^2dx\\
&\stackrel{\mathrm{\eqref{simple124}}}{\geq}&\frac{1}{6}\int_{\R^3}|\nabla u|^2dx
\end{eqnarray}
which implies that $m_V\geq 0$. Before we rule out $m_V =0$, we claim that there exists a constant $\varrho>0$ such that
\[
\int_{\R^3}|\nabla u|^2dx\geq \varrho>0,\ \  \forall u\in \mathcal{S}_{V }.
\]
Indeed, using $\langle I^{\prime}_{V }(u),u \rangle=0$, one has
\begin{eqnarray*}
\int_{\R^3}|\nabla u|^2+|u|^2dx&=&\int_{\R^3} \phi_u|u|^5dx+\int_{\R^3}uf(u)dx \\
  &\stackrel{\mathrm{\eqref{Sobolev0}}}{\leq}&  S^{-1}\bigg(\int_{\R^3}|  u|^6dx\bigg)^{\frac{5}{3}}+\epsilon\int_{\R^3}|u|^2dx+ C_\epsilon\int_{\R^3}|u|^6dx \\
   &\stackrel{\mathrm{\eqref{growth}}}{\leq} &  S^{-6}\bigg(\int_{\R^3}|\nabla u|^2dx\bigg)^5+\epsilon\int_{\R^3}|u|^2dx+ S^{-3}C_\epsilon\bigg(\int_{\R^3}|\nabla u|^2dx\bigg)^3.
\end{eqnarray*}
Let $\epsilon=1$ in the above formula, we can infer that
\[
\int_{\R^3}|\nabla u|^2dx\leq  S^{-6}\bigg(\int_{\R^3}|\nabla u|^2dx\bigg)^5+C S^{-3}\bigg(\int_{\R^3}|\nabla u|^2dx\bigg)^3,
\]
which gives us that the claim is true, where we have used the fact that $|\nabla u|_2>0$ otherwise $u\equiv 0$ by \eqref{Sobolev} which is a contradiction to $u\in \mathcal{S}_{V }$.

Now we show $m_V>0$, argue it by contradiction and assume there exists a sequence $\{u_n\}\subset \mathcal{S}_{V }$ such that $I_{V }(u_n)\to 0$. Hence from \eqref{mm}, one has
\[
\int_{\R^3}|\nabla u_n|^2dx\to 0,
\]
which is a contradiction to $|\nabla u_n|_2^2\geq \varrho>0$. Thus $m_{V }>0$, the proof is complete.
\end{proof}
\vskip0.3cm
\begin{proof}[\textbf{Proof of Theorem \ref{maintheorem1}}]
Let $\{u_n\}$ be a minimizing sequence of $m_{V }=\inf_{u\in \mathcal{S}_{V }}I_{V }(u)$, that is, the sequence $\{u_n\}\subset H^1(\R^3)$ satisfies $I^{\prime}_{V }(u_n) =0$ and $I_{V }(u_n)\to m_{V }$. Similar to the arguments in the proofs of Lemma \ref{S} and Lemma \ref{m}, we can conclude that $\{u_n\}$ is a bounded $(PS)_{m_{V }}$ sequence of $I_{V }$. Then by virtue of Lemma \ref{PScondition}, there exists a function $u\in H^1(\R^3)$ such that
\[
I^{\prime}_{V }(u) =0\ \ \text{and}\ \ I_{V }(u)=m_{V }>0.
\]
As Step 3 in the proof of Theorem \ref{maintheorem2}, we have that $u(x)>0$ in $\R^3$.
It is therefore that $(u,\phi_u)\in H^1(\R^3)\times D^{1,2}(\R^3)$ is a positive least energy solution to the system \eqref{mainequation1}. The proof is complete.
\end{proof}
\vskip0.3cm
\begin{proof}[\textbf{Proof of Corollary \ref{corollary}}]
The proof is totally same as that of Theorem \ref{maintheorem1} except \eqref{gggggg}, hence we just show $I_{V, \la}(u_0)\geq 0$ when the assumption $(V_1)$ is replaced by $(V_3)$. In fact, using $(V_3)$ and \eqref{Hardy}, one has
\begin{eqnarray*}
 |\overline{\beta}|&=&\bigg|\int_{\R^3}(x,\nabla V)|u_0|^2dx\bigg| \leq A\int_{\R^3}\frac{|u_0|^2}{|x|^2}dx\\
  &\leq&4A\int_{\R^3}|\nabla u_0|^2dx\leq \alpha.
\end{eqnarray*}
Hence similar to the proof of \eqref{gggggg},
\begin{eqnarray*}
  I_{V, \lambda}(u_0)&=&I_{V, \lambda}(u_0)- \frac{1}{3}P_{V, \lambda}(u_0)\\
 &=&\frac{1}{3}\alpha-\frac{1}{6}\overline{\beta}+\frac{1}{15}\delta
 \geq  \frac{1}{6} \alpha,
\end{eqnarray*}
which yields that $I_{V, \lambda}(u_0)\geq 0$. The proof is complete.
\end{proof}

\textbf{Acknowledgements:} The authors are supported by NSFC (Grant No. 11371158),
the program for Changjiang Scholars and Innovative Research Team in University (No.
IRT13066).

\end{document}